\documentclass[11pt,reqno]{amsart}
\usepackage{hyperref,amsmath,amsfonts,mathscinet,amssymb,amsthm}
\usepackage[mathcal]{euscript}
\usepackage{enumitem}
\usepackage{xcolor}
\usepackage{graphics}
\usepackage{mathrsfs}

\usepackage{todonotes}

\usepackage[normalem]{ulem}
\usepackage{soul}
\usepackage{cancel}
\newcommand\Del[1]{{\color{red}\ifmmode\cancel{#1}\else\sout{#1}\fi}}

\setlist[enumerate]{label={\rm(\roman*)}}

\theoremstyle{plain}
\newtheorem{thm}{Theorem}
\newtheorem{proposition}[thm]{Proposition}

\theoremstyle{definition}
\newtheorem{defn}[thm]{Definition}

\newtheorem{rem}[thm]{Remark}
\newtheorem{example}[thm]{Example}

\newtheorem{notation}[thm]{Notation}
\numberwithin{thm}{section}
\numberwithin{equation}{section}
\expandafter\let\expandafter\oldproof\csname\string\proof\endcsname
\let\oldendproof\endproof
\renewenvironment{proof}[1][\proofname]{%
  \oldproof[\bf #1]%
}{\oldendproof}

\def\M{\mathscr M}

\def\R{\mathbb R}

\newcommand{\aaa}{\left\langle\alpha\right\rangle}
\newcommand{\ppp}{\left\langle p\right\rangle}
\newcommand{\RR}{\mathscr{R}}
\newcommand{\abs}[1]{\left|{#1}\right|}
\def\esssup{\operatornamewithlimits{ess\,\sup}}

\newtoks\by
\newtoks\paper
\newtoks\book
\newtoks\jour
\newtoks\yr
\newtoks\pages
\newtoks\vol
\newtoks\publ
\newtoks\eds
\newtoks\proc
\newtoks\no
\def\ota{{\hbox{???}}}
\def\cLear{\by=\ota\paper=\ota\book=\ota\jour=\ota\yr=\ota
\pages=\ota\vol=\ota\publ=\ota}
\def\endpaper{\the\by, \textit{\the\paper},
{\the\jour} \textbf{\the\vol} (\the\yr), \the\pages.\cLear}
\def\endbook{\the\by, \textit{\the\book}, \the\publ.\cLear}
\def\endprep{\the\by, \textit{\the\paper}, \the\jour.\cLear}
\def\endproc{\the\by, \textit{\the\paper}, \the\publ, \the\pages.\cLear}
\def\name#1#2{#1 #2}
\def\et{ and }

\begin{document}

\title{Basic functional properties of certain
scale of rearrangement-invariant spaces}

\author[Hana Tur\v cinov\'a]{Hana Tur\v cinov\'a}

\email[H. Tur\v cinov\'a]{turcinova@karlin.mff.cuni.cz}

\address{
 Department of Mathematical Analysis,
 Faculty of Mathematics and Physics,
 Charles University,
 Sokolovsk\'a 83,
 186~00 Praha~8,
 Czech Republic}

\keywords{rearrangement-invariant spaces, maximal non-increasing rearrangement,}

\thanks{This research was supported by the SFG grant of Faculty of Mathematics and Physics, Charles University, grant SVV-2020-260583 and grant no.
P201-18-00580S of the Grant Agency of the Czech Republic.}

\begin{abstract}
Let $X$ be a rearrangement-invariant space over a non-atomic $\sigma$-finite measure space $(\RR,\mu)$ and let $\alpha\in(0,\infty)$. We define the functional
\begin{equation*}
    \|f\|_{X^{\aaa}} = \|((|f|^\alpha)^{**})^{\frac{1}{\alpha}}\|_{\overline{X}(0,\mu(\RR))},
\end{equation*}
in which $f$ is a~$\mu$-measurable scalar function defined on $(\RR,\mu)$ and $\overline{X}(0,\mu(\RR))$ is the representation space of $X$. We denote by $X^{\aaa}$ the collection of all almost everywhere finite
functions $f$ such that $\|f\|_{X^{\aaa}}$ is finite. These spaces recently surfaced in~\cite{CPS-frostman-1} and~\cite{CPS-frostman-2} in connection of optimality of target function spaces in general Sobolev embeddings involving upper Ahlfors regular measures.

We present a variety of results on these spaces including their basic functional properties, their relations to customary function spaces and mutual embeddings and, in a particular situation, a characterization of their associate structures. We discover a~new
one-parameter path of function spaces leading from a Lebesgue space to a Zygmund class and we compare it to the classical one.
\end{abstract}

\date{\today}

\maketitle

\bibliographystyle{alpha}

\section{Introduction}

Function spaces based on symmetrization have proved to be very useful in many branches of mathematics. They are often being used for fine-tuning of more robust classical function spaces originally based on different principles, such as Lebesgue or Orlicz spaces, in situations where these classes of spaces do not provide all answers. There is a vast literature available on the subject, for some recent advances see e.g.~\cite{Ta1,ACS} or~\cite{CPS}.


In the very recent papers~\cite{CPS-frostman-1} and~\cite{CPS-frostman-2}, Sobolev embeddings of arbitrary order have been considered into function spaces on subdomains of $\mathbb R^{n}$ endowed with upper Ahlfors regular measures, called
sometimes in the literature also Frostman measures, whose decay on balls is dominated by a certain power of their radii. The authors approached the problem from a new angle, combining the classical idea of deducing the
highly-dimensional inequalities from one-dimensional ones with a completely new interpolation technique involving a~logarithmically convex combination of two integral operators. Compared to other occurrences of reduction
principles that had been used in earlier work, the piece of information obtained from interpolation in~\cite{CPS-frostman-1} turned out to be somewhat mysterious and it took some further technical constructions to nail down correct target
classes in the Sobolev embeddings. The idea was further developed in~\cite{CPS-frostman-2}, where numerous examples involving Lorentz--Sobolev spaces and Orlicz--Sobolev spaces were presented. Embeddings involving upper Ahlfors regular measures of this generality have a number of important applications in measure theory, harmonic analysis, and theory of function spaces. Their most notorious example is the Hausdorff measure of a subdomain, leading thereby to a
wide variety of general Sobolev trace embedding theorems including many classical results (cf., for instance, \cite{Adams2,Adams3,Mazya543,Mazya548,CP-Trans}).

An interesting phenomenon was spotted in~\cite{CPS-frostman-1}, where it was shown that there is a huge difference between the so-called \textit{fast-decaying measures}, in the description of which the radii of balls are raised to a large power, and the \textit{slow-decaying measures}, for which the same power is small. The threshold between these cases is given by a balance condition depending on the dimension, the power determining the speed of the decay, and the order of the embedding. The case of the slow decaying measures is
the most interesting one and at the same time the most difficult one. In this case, the classical approaches fail and a new type of reduction principle for a Sobolev embedding involving a
slowly-decaying upper Ahlfors regular measure has to be found. Two techniques are developed in~\cite{CPS-frostman-1} to achieve this, each based on a fine work with a certain scale of function structures,
which we shall now describe in detail.

The point of departure is the triple consisting of a $\sigma$-finite non-atomic measure space $(\RR,\mu)$ with $\mu(\RR)\in(0,\infty]$, fixed once for all, a rearrangement-invariant space $X$ containing $\mu$-measurable scalar
functions defined on $\RR$, and a parameter $\alpha\in(0,\infty)$. By $\overline{X}(0,\mu(\RR))$ we denote the representation space of $X$ (which is known to exist, and in fact is even unique - see~\cite[Chapter~2, Theorem~4.10]{BS}).
We define two new structures:
\begin{itemize}
\item the collection $X^{\left\{\alpha\right\}}$ of all $\mu$-measurable and $\mu$-a.e.~finite scalar functions $f$ on $\RR$ for which the value
$\|f\|_{X^{\left\{\alpha\right\}}}$, defined as $\left\|\abs{f}^{\alpha}\right\|_{X}^{\frac{1}{\alpha}}$, is finite, and
\item the collection $X^{\aaa}$ of all $\mu$-measurable and $\mu$-a.e.~finite scalar functions $f$ on $\RR$ for which the value
$\|f\|_{X^{\aaa}}$, defined as $\|((|f|^\alpha)^{**})^{\frac{1}{\alpha}}\|_{\overline{X}(0,\mu(\RR))}$, is finite.
\end{itemize}
The scales $X^{\left\{\alpha\right\}}$ and $X^{\aaa}$ are used in~\cite{CPS-frostman-1} and~\cite{CPS-frostman-2} in two independent constructions in order to describe function spaces appearing as sharp target spaces in trace embedding theorems and in Sobolev embeddings involving upper Ahlfors regular measures. Each scale plays a specific role in an appropriate interpolation scheme. Now, while the spaces $X^{\left\{\alpha\right\}}$ have been treated before (cf.~e.g.~\cite{MaPe}), the spaces $X^{\aaa}$ are completely new. However, the authors of~\cite{CPS-frostman-1} and~\cite{CPS-frostman-2} apply these structures to their purposes without dwelling on their basic functional properties.
The importance of the spaces $X^{\aaa}$, which stems from the mentioned applications, shows that it would be desirable to obtain some more advanced information about them. Our aim in this paper is precisely to carry out such a study.

We will subsequently focus on several important questions concerning these spaces. We approach the spaces through their governing functionals. Given a functional $\varrho\colon\M_+(\RR,\mu)\to[0,\infty]$, where $\M_+(\RR,\mu)$ is the set of all scalar-valued $\mu$-measurable functions with values in $[0,\infty]$, we
define the subset $X=X_{\varrho}$ of $\M_0(\RR,\mu)$ as the collection of all scalar-valued $\mu$-measurable functions on $\RR$ which are finite $\mu$-a.e.~on
$\RR$ and such that $\|f\|_{X}=\varrho(|f|)<\infty$.
Then we define the space $X^{\aaa}$ through the appropriate governing functional, denoted $\varrho^{\aaa}$.

We shall first concentrate on the question when, given $\varrho$ and $\alpha$, the functional $\varrho^{\aaa}$ satisfies individual axioms of a rearrangement-invariant Banach function norm. Surprisingly, this task turns out to be rather complicated. We will then proceed to characterizing the fundamental function of $\varrho^{\aaa}$ (or, which is the same, of $X^{\aaa}$). We will investigate mutual relations of the structures having the form $X^{\aaa}$,
mostly expressed as sharp embeddings, and also their relations to other customary scales of function spaces.
We will also
characterize their associate structures in the special case when the underlying space $X$ is the classical Lorentz space of type $\Lambda$. We will provide non-trivial examples illustrating the general results.

Let us now describe the structure of the paper.

In Section~\ref{S:preliminaries}, we collect all the necessary background material. We fix definitions here, and also most of the notation. Section~\ref{S:basic} is devoted to the study of the basic properties of the
functionals $\varrho^{\aaa}$. We carry out a thorough scrutiny of these functionals within the context of axioms of rearrangement-invariant (quasi)norms. It turns out that the question of the validity of these axioms for   $\varrho^{\aaa}$ is quite non-trivial and contains some rather concealed pitfalls. Having established basic knowledge, we continue to characterize fundamental functions, a property which often contains a key information when dealing with rearrangement-invariant structures. We finally characterize when $\varrho^{\aaa}$ is a rearrangement-invariant~norm and illustrate the general results obtained on examples.

In Section~\ref{S:embeddings}, we focus on mutual relations between the spaces $X^{\aaa}$. The results are mostly formulated either as norm inequalities or continuous embeddings. We point out that the two
above-mentioned scales are linked in an interesting way. While $X^{\aaa}$ is always continuously embedded into $X$, the converse is true if and only if the Hardy averaging operator is bounded on the subcone of non-increasing
functions of the representation space of $X^{\left\{\frac{1}{\alpha}\right\}}$. Among corresponding examples we discover a~new type of a function space.

In Section~\ref{S:bridge}, we present two independent ways of bridging a gap between a Lebesgue space and the related, slightly smaller Zygmund class, both over the same finite measure space. We employ here the new spaces from Section~\ref{S:embeddings} to construct a one-parameter path of function spaces bridging the two mentioned spaces and compare it to the
natural well-known one. We provide a comprehensive information about all possible embeddings between spaces belonging to both scales.

In the last section, we study associate structures of $X^{\aaa}$ when $X$ is a classical Lorentz space. For this purpose, we drop the requirement that $X$
has to be a
rearrangement-invariant~space. This relaxation is possible due to the special technical nature of classical Lorentz spaces. We reduce the problem to a question of quantifying the operator norm of a certain continuous embedding. A key idea of this technique is that a~generic function plays, for a time being, the role of a weight.

\section{Preliminaries} \label{S:preliminaries}

In this section, we collect definitions of objects of our study, fix notation and give a survey of concepts and results from functional analysis and theory of function spaces that will be used in the subsequent parts of the paper. Our standard general reference is \cite{BS} and \cite{PKJF}, where more details can be found.

Let $(\mathscr{R},\mu)$ be a~non-atomic $\sigma$-finite measure space with $\mu(\RR)\in(0,\infty]$. We denote by $\mathscr{M}(\mathscr{R},\mu)$ the set of all $\mu$-measurable functions on $\mathscr{R}$ whose values lie
in $[-\infty,\infty]$, by $\mathscr{M}_+(\mathscr{R},\mu)$ the set of all functions in $\mathscr{M}(\RR,\mu)$ whose values lie in $[0,\infty]$, and by $\mathscr{M}_{0}(\mathscr{R},\mu)$ the set of all functions
in $\mathscr{M}(\mathscr{R},\mu)$ that are finite $\mu$-a.e.~on $\mathscr{R}$. By $\mathscr{M}_+(0,\mu(\mathscr{R}))$ we denote the set of all $m$-measurable functions on the interval $(0,\mu(\mathscr{R}))$ whose values lie in $[0,\infty]$,
where $m$ denotes the one-dimensional Lebesgue measure. We use the symbol $\mathscr{M}_{0}(0,\mu(\mathscr{R}))$ in an analogous way.

For $u\in\mathscr{M}(\mathscr{R},\mu)$, the function $u^*\colon [0,\mu(\RR))\to[0,\infty]$, defined by
\begin{equation*}
    u^{\ast}(t)=\inf\{\lambda\geq 0:\mu(\{x\in \RR:\abs{u(x)}>\lambda\})\leq t\}\quad\text{for $t\in[0,\mu(\RR))$,}
\end{equation*}
is called the \emph{non-increasing rearrangement} of $u$. The function $u^{**}\colon(0,\mu(\mathscr{R}))\rightarrow[0,\infty]$, defined by
$$
u^{**}(t)=\frac{1}{t}\int_0^t u^*(s)\,d s\quad\text{for $t\in(0,\mu(\mathscr{R}))$},
$$
is called the \emph{maximal non-increasing rearrangement} of $u$.

\begin{rem}\label{propsf**}
Assume that $u,v\in\mathscr{M}(\mathscr{R},\mu)$, $\left\{u_n\right\}_{n=1}^\infty \subset \mathscr{M}(\mathscr{R},\mu)$, $\alpha\in(0,\infty)$ and $a\in\R$. Then
\begin{itemize}
\item $u^*$, $u^{**}$ are non-negative and non-increasing on $(0,\mu(\mathscr{R}))$, $u^{**}$ is continuous,
\item $u^{*}\equiv0$ if and only if $u^{**}\equiv0$, which in turn holds if and only if $u=0$ $\mu$-a.e.~on $\mathscr{R}$,
\item if $0\leq v\leq u$ $\mu$-a.e.~on $\mathscr{R}$, then $v^{*}\leq u^{*}$ and $v^{**}\leq u^{**}$, $(au)^{*}=|a|u^{*}$, $(au)^{**}=|a|u^{**}$ and if $0\leq u_n\nearrow u$ $\mu$-a.e.~on $\mathscr{R}$, then $u_n^{**}\nearrow u^{**}$,
\item $u^*(t)\leq u^{**}(t)$ for $t\in(0,\mu(\mathscr{R}))$,
\item $(u^*)^{\alpha}=(|u|^{\alpha})^{*}$.
\end{itemize}
If, moreover, either $u,v\in\mathscr{M}_0(\mathscr{R},\mu)$ or $u,v\in\mathscr{M}_+(\mathscr{R},\mu)$, then
\begin{itemize}
\item $(u+v)^{**}(t)\le u^{**}(t)+v^{**}(t)$ for $t\in(0,\mu(\mathscr{R}))$,
\item $(u+v)^{*}(s+t)\le u^{*}(s)+v^{*}(t)$ for $s,t,s+t\in[0,\mu(\mathscr{R}))$.
\end{itemize}
\end{rem}

\begin{defn}[continuous embedding]
Let $X, Y$ be two quasinormed linear spaces and let $X\subset Y$.
We say that the space X is \emph{continuously embedded} into the space Y, denoted $X\hookrightarrow Y$, if there exists a constant $C\in(0,\infty)$ such that
$$
\left\|u\right\|_{Y}\leq C \left\|u\right\|_{X}\quad \text{for every $u\in \M(\RR,\mu)$}.
$$
\end{defn}

We denote by $\chi_E$ the characteristic function of a set $E$. If $A,B$ are two non-negative quantities, we write $A\lesssim B$ if there exists a positive constant $c$ independent of adequate parameters involved in $A$ and $B$ such that $A\leq c B$. When $A\lesssim B$ and simultaneously $B\lesssim A$, we write $A\approx B$. The convention $0\cdot \infty=0$ applies.

\begin{notation}\label{NOT:x-rho}
If $\varrho\colon \mathscr{M}_{+}(\RR,\mu)\rightarrow [0,\infty]$ is a functional, then we denote by $X_{\varrho}$ the collection of all functions $f\in\M_0(\RR,\mu)$ such that $\|f\|_{X_{\varrho}}=\varrho(|f|)<\infty$.
\end{notation}

If $\varrho\colon \mathscr{M}_{+}(\RR,\mu)\rightarrow [0,\infty]$ and $\sigma\colon \mathscr{M}_{+}(\RR,\mu)\rightarrow [0,\infty]$, then by $X_{\varrho}=X_{\sigma}$ we shall mean that $X_{\varrho}$ and $X_{\sigma}$ are equal in the set-theoretical sense and there exist positive
constants $c,C$ such that $c\varrho(f)\le\sigma(f)\le C\varrho(f)$ for every $f\in \mathscr{M}_{+}(\RR,\mu)$. In such case we shall say that $X_{\varrho}$ and $X_{\sigma}$ are \textit{equivalent}.

\begin{defn}
We say that a functional $\varrho\colon \mathscr{M}_{+}(\RR,\mu)\rightarrow [0,\infty]$ is a \emph{rearrangement-invariant norm} (an \textit{r.i.~norm} for short) if, for all $f$, $g$ and $\{f_n\}^{\infty}_{n=1}$ in $\mathscr{M}_{+}(\RR,\mu)$, for every $\lambda\in[0,\infty)$ and for every $\mu$-measurable subset $E$ of $\RR$, the following six properties are satisfied:

(P1) $\varrho(f)=0 \Leftrightarrow f=0$ $\mu$-a.e.~on $\mathscr{R}$; $\varrho(\lambda f)=\lambda\varrho(f)$; $\varrho(f+g)\leq
\varrho(f)+\varrho(g)$;

(P2) $g\leq f$ $\mu$-a.e.~on $\mathscr{R}$ $\Rightarrow\varrho(g)\leq\varrho(f)$;

(P3) $f_n\nearrow f$ $\mu$-a.e. on $\RR$ $\Rightarrow\varrho(f_n)\nearrow\varrho(f)$;

(P4) $\mu(E)<\infty \Rightarrow \varrho(\chi_{E})<\infty$;

(P5) $\mu(E)<\infty \Rightarrow \int_{E}f \,d\mu\leq C_E \varrho(f)$ for some constant $C_E \in (0,\infty)$ possibly depending on $E$ and $\varrho$ but independent of $f$;

(P6) $\varrho(f)=\varrho(g)$ whenever $f\sp*=g\sp*$.

\medskip

We say that $\varrho\colon \mathscr{M}_{+}(\RR,\mu)\rightarrow [0,\infty]$ is a \textit{rearrangement-invariant quasinorm} (an \textit{r.i.~quasinorm} for short) if it satisfies (P2), (P3), (P4) and (P6), and (P1) replaced by its weakened modification (Q1), where

(Q1) $\varrho(f)=0\Leftrightarrow f=0$ $\mu$-a.e.~on $\mathscr{R}$, $\varrho(\lambda f)=\lambda\varrho(f)$ and there exists $C\in(0,\infty)$ such that
\begin{equation*}
    \varrho(f+g)\le C(\varrho(f)+\varrho(g))\quad\text{for every $f,g\in\mathscr{M}_{+}(\RR,\mu)$.}
\end{equation*}
The infimum over all such constants $C$ is called the \textit{modulus of concavity} of $\varrho$ (cf.~\cite{C2F}).

For an~r.i.~quasinorm $\varrho$ and $X=X_{\varrho}$ we denote $\left\|f\right\|_{X}=\varrho(|f|)$ for $f\in\M(\RR,\mu)$. We then say that $X$ is a \emph{rearrangement-invariant quasi-Banach function space} (a \emph{quasi-r.i.~space} for short) over $(\RR,\mu)$. In case $\varrho$
is an r.i.~norm, we call $X$ a \emph{rearrangement-invariant Banach function space} (an \emph{r.i.~space} for short) over $(\RR,\mu)$.
\end{defn}

It is worth noticing that the expression $\|f\|_X$ is defined for every $f\in\mathscr{M}(\mathscr{R},\mu)$ (although it might be infinite) and that $\|f\|_X<\infty$ if and only if $f\in X$.

A pivotal example of an~r.i.~space is the Lebesgue space.

\begin{defn} Let $p\in(0,\infty]$. Then we define the functional $\varrho_p$ on $\M_+(\RR,\mu)$ by
\begin{equation*}
    \varrho_p(f)=
    \begin{cases}
    \left(\int_{\RR}f\sp p\,d\mu\right)\sp{\frac1p}\ &\textup{if}\ p\in(0,\infty),
        \\
    \mu\text{-}\operatorname{ess\,sup}_{x\in \RR} f(x) &\textup{if}\ p=\infty.
    \end{cases}
\end{equation*}
We shall denote $L^p=X_{\varrho_p}$.
\end{defn}

Note that $L^p$ is a quasi-r.i.~space for every $p\in(0,\infty]$ and it is an~r.i.~space if and only if $p\in[1,\infty]$.

\begin{defn}
Let $\varrho\colon \mathscr{M}_{+}(\RR,\mu)\rightarrow [0,\infty]$ be a~functional satisfying (P6). Then the function $\varphi\colon[0,\mu(\RR))\to[0,\infty]$ given by $\varphi(t)=\varrho(\chi_{E})$, where $E$ is any $\mu$-measurable subset of $\RR$ such that $\mu(E)=t$, is well defined, and will be called the \textit{fundamental function} of $\varrho$. We also say that $\varphi$ is the fundamental function of $X$, where $X=X_{\varrho}$.
\end{defn}

\begin{defn}\label{D:associate}
Let $\varrho\colon\M_+(\RR,\mu)\to[0,\infty]$ be a functional. Then we define another such functional, $\varrho'\colon\M_+(\RR,\mu)\to[0,\infty]$, by
\begin{equation*}
    \varrho'(f) = \sup\left\{\int_{\RR}fg\,d\mu :  g\in\M_+(\RR,\mu),\ \varrho(g)\le 1\right\}.
\end{equation*}
Then $\varrho'$ is called the \textit{associate functional} of $\varrho$. If $X=X_{\varrho}$, then we write $X'=X_{\varrho'}$. If $\varrho$ is an~r.i.~norm, then so is $\varrho'$, and it is called the \textit{associate norm} of $\varrho$ and $X'$ is called the \textit{associate space} of $X$.
\end{defn}

If $p\in[1,\infty]$, then $\varrho_{p}'=\varrho_{p'}$, where $p'$ is given by
\begin{equation*}
    p'=
        \begin{cases}
            \infty &\text{if $p=1$,}
                \\
            \frac{p}{p-1} &\text{if $p\in(1,\infty)$,}
                \\
            1  &\text{if $p=\infty$.}
        \end{cases}
\end{equation*}

If $\varrho\colon\M_+(\RR,\mu)\to[0,\infty]$ is a functional, then the \textit{H\"older inequality}
\begin{equation}\label{E:holder}
    \int_{0}^{\mu(\RR)}fg\,d\mu\le \varrho(f)\varrho'(g)
\end{equation}
holds for every $f,g\in\M_+(\RR,\mu)$ such that $\varrho(f)<\infty$ and $\varrho'(g)<\infty$. In the case when $\varrho$ is an r.i.~norm, the inequality~\eqref{E:holder} holds for any $f,g\in\M_+(\RR,\mu)$.

\begin{rem}
It was shown in~\cite[Theorem~3.1]{GS}, see also~\cite[Remark~2.3]{EKP}, that if a functional $\varrho\colon\M_+(\RR,\mu)\to[0,\infty]$ satisfies at least (P4) and (P5), then $\varrho'$ satisfies (P1)--(P5).
\end{rem}

Assume that $\varrho$ is an r.i.~norm over $(\RR,\mu)$. Then there exists a~uniquely defined r.i.~norm $\overline{\varrho}$ over $((0,\mu(\RR)), m)$ such that $\varrho(f)=\overline{\varrho}(f^{*})$ for every $f\in\M_+(\RR,\mu)$. We will call $\overline{\varrho}$ the \textit{representation norm} of $\varrho$. We denote $\overline{X}=X_{\overline{\varrho}}$ and we write $\|h\|_{\overline{X}}=\overline{\varrho}(|h|)$ for every $h\in\M(\RR,\mu)$.

\begin{defn}
We say that the functions $u,v\in\mathscr{M}(\RR,\mu)$ are in the \textit{Hardy--Littlewood--P\'olya relation}, a fact we denote by $u\prec v$, if, for every $t\in[0,\mu(\RR))$, one has
$$
\int_0^t u^*(s)\,ds\leq\int_0^t v^*(s)\,ds.
$$
\end{defn}

The~\textit{Hardy--Littlewood--P\'olya principle} states that if $u,v\in\mathscr{M}(\RR,\mu)$ satisfy $u\prec v$ and $\varrho\colon\M_+(\RR,\mu)\to[0,\infty]$ is an r.i.~norm, then
$$
\varrho(|u|)\leq\varrho(|v|).
$$

Although Lebesgue spaces play a~primary role in analysis, there are other scales of function spaces that are also of interest. Lebesgue spaces have been generalized in many ways, two of the most important ones being represented by Lorentz spaces and Orlicz spaces. We shall recall definitions and basic properties of these spaces. For proofs and more details see~\cite{BS} or~\cite{PKJF}.

\begin{defn}
Assume that $p,q\in(0,\infty]$. We define the functionals $\varrho_{p,q}$ and $\varrho_{(p,q)}$ on $\M_+(\RR,\mu)$ by
\begin{equation*}
\varrho_{p,q}(f)=
\overline{\varrho_q}\left(s\sp{\frac{1}{p}-\frac{1}{q}}f^*(s)\right)
\quad \hbox{and} \quad \varrho_{(p,q)}(f)=
\overline{\varrho_q}\left(s\sp{\frac{1}{p}-\frac{1}{q}}f^{**}(s)\right).
\end{equation*}
We shall denote $L\sp{p,q}=X_{\varrho_{p,q}}$ and $L\sp{(p,q)}=X_{\varrho_{(p,q)}}$. These spaces (both types) are called \textit{Lorentz spaces}.
\end{defn}

\noindent Obviously (cf. Remark~\ref{propsf**}), one has $L\sp{(p,q)}\hookrightarrow L\sp{p,q}$ for any choice of $p,q$. Moreover,
\begin{equation*}
L\sp{p,q}=L\sp{(p,q)}\quad\textup{if}\ p\in(1,\infty].
\end{equation*}
It will be useful to recall that $L\sp{p,p}=L\sp p$ for every $p\in(0,\infty]$ and that $L\sp{p,q}\hookrightarrow L\sp {p,r}$ whenever $p\in(0,\infty]$ and $0< q\leq r\leq \infty$. If either $p\in(0,\infty)$ and $q\in(0,\infty]$ or $p=q=\infty$, then $L^{p,q}$ is a quasi-r.i.~space. If one of the  conditions
\begin{equation*}
\begin{cases}
p\in(1,\infty),\ q\in[1,\infty],\\
p=q=1,\\
p=q=\infty,
\end{cases}
\end{equation*}
holds, then $L\sp{p,q}$ is equivalent to an~r.i.~space.

\begin{rem}
A quasi-r.i.~space may, or may not, satisfy (P5). A typical example of a quasi-r.i.~space which does not satisfy (P5) is $L^{p}$ with $p\in(0,1)$. On the other hand a typical example of a quasi-r.i.~space which is not an~r.i.~space but satisfies (P5) nevertheless is the Lorentz space $L^{p,q}$ with $p\in(1,\infty)$ and $q\in(0,1)$.
\end{rem}

\begin{defn}
Let $\mu(\RR)<\infty$, let $p,q\in(0,\infty]$ and let $\alpha \in\R$. We define the functionals $\varrho_{p,q;\alpha}$ and $\varrho_{(p,q;\alpha)}$ on $\M_+(\RR,\mu)$ by
\begin{equation*}
\begin{cases}
\varrho_{p,q;\alpha}(f)=
\overline{\varrho_q}\left(s\sp{\frac{1}{p}-\frac{1}{q}}\log \sp
\alpha\left(\tfrac{e\mu(\RR)}{s}\right) f^*(s)\right),\\
\varrho_{(p,q;\alpha)}(f)=
\overline{\varrho_q}\left(s\sp{\frac{1}{p}-\frac{1}{q}}\log \sp
\alpha\left(\tfrac{e\mu(\RR)}{s}\right) f^{**}(s)\right).
\end{cases}
\end{equation*}
We shall denote $L\sp{p,q;\alpha}=X_{\varrho_{p,q;\alpha}}$ and $L\sp{(p,q;\alpha)}=X_{\varrho_{(p,q;\alpha)}}$. We call $L\sp{p,q;\alpha}$ and $L\sp{(p,q;\alpha)}$ \textit{Lorentz--Zygmund spaces}.
\end{defn}

If one of the following
conditions
\begin{equation*}
\begin{cases}
p\in(1,\infty),\ q\in[1,\infty],\\
p=1,\ q=1,\ \alpha \in[0,\infty),\\
p=\infty,\ q=\infty,\ \alpha\in(-\infty,0],\\
p=\infty,\ q\in[1,\infty),\ \alpha\in(-\infty,-\frac1q),
\end{cases}
\end{equation*}
is satisfied, then $L^{p,q;\alpha}$ is equivalent
to an r.i.~space. Lorentz--Zygmund spaces were introduced in~\cite{BR}, further details can be found for instance in~\cite{glz} or~\cite{EGO}. They contain many interesting nontrivial function spaces which have important applications, mainly in various limiting or critical situations, see e.g.~\cite{BW}.

\begin{defn}
Let $p\in(0,\infty]$ and let $w$ be a weight on $(0,\mu(\RR))$ (that is, $w\in\M_+(0,\mu(\RR))$). Then we define the functionals $\varrho_{\Lambda^p(w)}$ and $\varrho_{\Gamma^p(w)}$ on $\M_+(\RR,\mu)$ by
\begin{align*}
\varrho_{\Lambda^p(w)}(f)=
\begin{cases}
\left(\int_0^{\mu(\RR)}f^*(t)^pw(t)\,dt\right)^{\frac{1}{p}}\quad&\text{if $p\in(0,\infty)$},\\
m\text{-}\esssup_{t\in(0,\mu(\RR))}f^*(t)w(t)\quad&\text{if $p=\infty$}
\end{cases}
\end{align*}
and
\begin{align*}
\varrho_{\Gamma^p(w)}(f)=
\begin{cases}
\left(\int_0^{\mu(\RR)}f^{**}(t)^pw(t)\,dt\right)^{\frac{1}{p}}\quad&\text{if $p\in(0,\infty)$},\\
m\text{-}\esssup_{t\in(0,\mu(\RR))}f^{**}(t)w(t)\quad&\text{if $p=\infty$.}
\end{cases}
\end{align*}
We write $\Lambda^p(w)=X_{\varrho_{\Lambda^p(w)}}$ and $\Gamma^p(w)=X_{\varrho_{\Gamma^p(w)}}$. Then $\Lambda^p(w)$ and $\Gamma^p(w)$ are called \textit{classical Lorentz spaces}.
\end{defn}

The question of the (quasi)normability of classical Lorentz spaces is complicated, details can be found either scattered in literature or surveyed in~\cite[Section 10.2]{PKJF}. Note that $\Lambda^{q}(t^{\frac{q}{p}-1})=L^{p,q}$ and $\Gamma^{q} (t^{\frac{q}{p}-1})=L^{(p,q)}$.

A generalization of Lebesgue spaces in a different direction
is provided by \textit{Orlicz spaces}.

\begin{defn}
We say that a function~$A\colon[0,\infty]\to[0,\infty]$ is a \emph{Young
function} if it is left continuous, increasing and convex on $[0,\infty]$, satisfying
$A(0)=0$, and such that $A$ is not constant in $(0,\infty)$. Then
\begin{equation*}
A(t) = \int _0^t a(\tau ) d\tau \quad \text{for $t \in[0,\infty]$},
\end{equation*}
for some non-decreasing, left-continuous function $a\colon [0, \infty )
\to [0, \infty ]$ which is neither identically equal to $0$, nor to
$\infty$. We then define the functional $\varrho_A$ on $\M_+(\RR,\mu)$ by
\begin{equation*}
\varrho_A(f)= \inf \left\{ \lambda \in(0,\infty) :  \int_{\RR} A \left(\frac{f}{\lambda} \right) d\mu \leq 1 \right\}.
\end{equation*}
The corresponding space $L^A=X_{\varrho_A}$ is called the \textit{Orlicz space}.
\end{defn}

In particular, $L^A = L^p$ if $A(t)= t^p$ for some $p \in [1, \infty )$, and
$L^A = L^\infty $ if $A= \infty
\chi_{(1, \infty)}$.

If $\mu(\RR)<\infty$ and either $p\in(1,\infty)$ and $\alpha \in \R$ or $p=1$ and $ \alpha \in[0,\infty)$, then we denote by $L^p(\log L)^{\alpha}$ the Orlicz space
associated with a Young function equivalent to $t^p (\log t)^{\alpha p}$ near infinity. If $\beta\in(0,\infty)$, then we will denote by $\exp L^\beta$ the Orlicz space built upon a Young
function equivalent to $e^{t^\beta}$ near infinity. Let us recall that $L^p(\log L)^{\alpha}=L^{p,p;\alpha}$ and $\exp L^\beta=L^{\infty,\infty,-\frac{1}{\beta}}$.

We will now define a specific functional, built upon a given norm, that will be useful in the sequel.

\begin{defn}
Let $\varrho$ be a rearrangement-invariant norm over $(\RR,\mu)$ and let $\alpha\in(0,\infty)$. We then define the functional $\varrho^{\left\{\alpha\right\}}$ on $\M_+(\RR,\mu)$ by
\begin{equation*}
    \varrho^{\left\{\alpha\right\}}(f) = \varrho(f^{\alpha})^{\frac{1}{\alpha}}.
\end{equation*}
If $X=X_{\varrho}$, then we denote $X^{\left\{\alpha\right\}}=X_{\varrho^{\left\{\alpha\right\}}}$.
\end{defn}

It will be useful to note that if $b\in(0,\infty)$ and $h\in\M_+(0,b)$ is non-increasing, then
\begin{equation}\label{E:rozpulil-pess}
    \int_{0}^{b}h(t)\,dt \le 2\int_{0}^{\frac{b}{2}}h(t)\,dt.
\end{equation}

\section{Basic functional properties of the scale}\label{S:basic}

In this section, we introduce the scale of function spaces which constitutes the main object of study in this paper.

\begin{defn}
Let $\varrho$ be a rearrangement-invariant norm over $(\RR,\mu)$ and let $\alpha\in(0,\infty)$. We then define the functional $\varrho^{\aaa}$ on $\M_+(\RR,\mu)$ by
\begin{equation*}
    \varrho^{\aaa}(f) = \overline{\varrho}\left((\left(f^{\alpha}\right)^{**})^{\frac{1}{\alpha}}\right).
\end{equation*}
If $X=X_{\varrho}$, then we denote $X^{\aaa}= X_{\varrho^{\aaa}}$ in accordance with Notation~\ref{NOT:x-rho}.
\end{defn}

We begin by observing a basic relation between $X$ and $X^{\aaa}$.

\begin{proposition}\label{XaembX}
Let $\varrho$ be a rearrangement-invariant norm over $(\RR,\mu)$ and let $\alpha\in(0,\infty)$. Then $\varrho(f)\le\varrho^{\aaa}(f)$ for every $f\in\M_+(\RR,\mu)$.
\end{proposition}

\begin{proof}
Since $\varrho$ is an r.i.~norm, we know that there exists its unique representation norm $\overline{\varrho}$. Thus, using the properties of rearrangements collected in Remark~\ref{propsf**}, elementary inequalities and (P2) for $\overline{\varrho}$, we get, for every $f\in\M(\RR,\mu)$,
\begin{align*}
  \varrho(f) & = \overline{\varrho}(f^{*}) = \overline{\varrho}(((f^{*})^{\alpha})^{\frac{1}{\alpha}}) = \overline{\varrho}(((f^{\alpha})^{*})^{\frac{1}{\alpha}})
    \le \overline{\varrho}(((f^{\alpha})^{**})^{\frac{1}{\alpha}}) = \varrho^{\aaa}(f),
\end{align*}
and the assertion follows.
\end{proof}

Our next aim is to investigate when, for a given r.i.~norm $\varrho$, the functional $\varrho^{\aaa}$ is at least a quasinorm.

\begin{thm}\label{quasinorm}
Let $\varrho$ be a rearrangement-invariant norm over $(\RR,\mu)$ and let $\alpha\in(0,\infty)$. Then $\varrho^{\aaa}$ satisfies \textup{(P2)}, \textup{(P3)}, \textup{(P6)} and \textup{(Q1)}.
\end{thm}

\begin{proof}
First, note that $\varrho^{\aaa}$ obviously satisfies (P6). Since $\varrho$ is an r.i.~norm, we know that there exists its unique representation norm $\overline{\varrho}$. Properties (P2) and (P3) for $\varrho^{\aaa}$ follow readily from their counterparts for $\overline{\varrho}$, using also elementary properties of powers and rearrangements (see Remark~\ref{propsf**}). To prove (Q1) for $\varrho^{\aaa}$, let us first show the positive homogeneity. Let $f\in\M_+(\RR,\mu)$ and $c\ge0$. Using the positive homogeneity of the maximal non-increasing rearrangement mentioned in Remark~\ref{propsf**} and of $\overline{\varrho}$, we have
\begin{align*}
\varrho^{\aaa}(cf)
&=
\overline{\varrho}\left(\left(\left((cf)^\alpha\right)^{**}\right)^{\frac1{\alpha}}\right)
=
c\overline{\varrho}\left(\left(\left(f^\alpha\right)^{**}\right)^{\frac1{\alpha}}\right)
= c \varrho^{\aaa}(f).
\end{align*}
Now let us turn our attention to the subadditivity of $\varrho^{\aaa}$. Let $f,g \in \M_+(\RR,\mu)$. We shall distinguish two cases according to whether $\alpha\in(0,1)$ or $\alpha\in[1,\infty)$ and treat them separately. Assume first that $\alpha\in[1,\infty)$. Then, using the definition of $\varrho^{\aaa}$, elementary inequalities, subadditivity, monotonicity and positive homogeneity of the maximal non-increasing rearrangement, (P1) and (P2) for $\overline{\varrho}$, we get
\begin{align*}
\varrho^{\aaa}(f+g)
& =
\overline{\varrho}\left(\left(\left((f+g)^\alpha\right)^{**}\right)^{\frac1{\alpha}}\right)
   \leq
\overline{\varrho}\left(\left(\left(2^{\alpha-1}(f^\alpha+g^\alpha)\right)^{**}\right)^{\frac1{\alpha}}\right)
    \\
&\leq
2^{1-\frac{1}{\alpha}}\overline{\varrho}\left(\left(\left(f^\alpha\right)^{**}\right)^{\frac1{\alpha}}
+\left(\left(g^\alpha\right)^{**}\right)^{\frac1{\alpha}}\right)
    \\
&\leq
2^{1-\frac{1}{\alpha}}\left(\varrho^{\aaa}(f)+\varrho^{\aaa}(g)\right).
\end{align*}
Now assume that $\alpha\in(0,1)$. Then, using analogous principles, we obtain
\begin{align*}
\varrho^{\aaa}(f+g)
& =
\overline{\varrho}\left(\left(\left((f+g)^\alpha\right)^{**}\right)^{\frac1{\alpha}}\right)
\leq
\overline{\varrho}\left(\left(\left(f^\alpha\right)^{**}+\left(g^\alpha\right)^{**}\right)^{\frac1{\alpha}}\right)
    \\
&\leq
\overline{\varrho}\left(2^{\frac{1}{\alpha}-1}\left(\left(\left(f^\alpha\right)^{**}\right)^{\frac1{\alpha}}+\left(\left(g^\alpha\right)^{**}\right)^{\frac1{\alpha}}\right)\right)
    \\
&\leq
2^{\frac{1}{\alpha}-1}\left(\varrho^{\aaa}(f)+\varrho^{\aaa}(g)\right).
\end{align*}
Altogether, we get, in each case,
\begin{equation*}
    \varrho^{\aaa}(f+g)\leq2^{\left|\frac{1}{\alpha}-1\right|}\left(\varrho^{\aaa}(f)+\varrho^{\aaa}(g)\right).
\end{equation*}
It remains to notice that it follows from Remark~\ref{propsf**} and (P1) for $\overline{\varrho}$ that, for $f\in\M_+(\RR,\mu)$,
\begin{align*}
\varrho^{\aaa}(f)=0\ \Leftrightarrow\
\overline{\varrho}\left(\left(\left(f^\alpha\right)^{**}\right)^{\frac1{\alpha}}\right)=0\
\Leftrightarrow\
 f=0 \quad \mu\text{-a.e.~on $\mathscr{R}$.}
\end{align*}
The proof is complete.
\end{proof}

\begin{rem}
Under the assumptions of Theorem~\ref{quasinorm}, the functional $\varrho^{\aaa}$ may, or may not, satisfy (P4). We shall come back to this question in more detail in Theorem~\ref{T:P4} below.
\end{rem}

Our next aim is to point out that much more can be said when $\alpha\ge 1$. To this end we shall need the following useful general principle of independent interest.

\begin{proposition}\label{P:maly}
Let $h\colon(0,\mu(\RR))\rightarrow[0,\infty)$ be right-continuous and non-increasing. Then, for every fixed $t\in(0,\mu(\RR))$, the operator
$$
u\longmapsto\int_0^{t} u^*(s)h(s)\,d s
$$
is subadditive both on $\mathscr{M}_+(\mathscr{R},\mu)$ and on $\mathscr{M}_0(\mathscr{R},\mu)$.
\end{proposition}

\begin{proof}
Fix $t\in(0,\mu(\RR))$ and let us represent $h$ as
$$
h(s)=c+\int_s^{\mu(\RR)} \,d\nu \quad\text{for $s\in(0,\mu(\RR))$,}
$$
where c is a constant and $\nu$ is a Lebesgue-Stieltjes measure on $(0,{\mu(\RR)})$. Then, by the Fubini theorem, we have
\begin{align*}
    \int_0^{t} u^*(s)h(s)\,d s
    &= c \int_0^{t} u^*(s)\,d s + \int_0^{t} u^*(s)\int_s^{t} \,d \nu(y) \,d s + \int_0^{t} u^*(s)\int_t^{\mu(\RR)} \,d \nu(y) \,d s
        \\
    &= c \int_0^{t} u^*(s)\,d s + \int_0^{t} \int_{0}^{y}u^*(s)\,ds \,d \nu(y)  + \int_t^{\mu(\RR)} \,d \nu(y)\int_0^{t} u^*(s)\,d s
        \\
    &= \left(c+\int_t^{\mu(\RR)} \,d \nu(y)\right)tu^{**}(t) + \int_0^{t} yu^{**}(y)\,d \nu(y),
\end{align*}
which is subadditive by Remark~\ref{propsf**}.
\end{proof}

\begin{thm}\label{norm}
Let $\varrho$ be a rearrangement-invariant norm over $(\RR,\mu)$ and let $\alpha\in[1,\infty)$. Then $\varrho^{\aaa}$ satisfies \textup{(P1)}.
\end{thm}

\begin{proof}
In view of Theorem~\ref{quasinorm} it only remains to show that $\varrho^{\aaa}$ satisfies the triangle inequality. This is clear if $\alpha=1$, so let us assume that $\alpha\in(1,\infty)$. Suppose that $f,g\in \M_+(\RR,\mu)$ and fix $t\in(0,\mu(\RR))$. If $f+g$ is equal to zero $\mu$-almost everywhere on $\RR$, then there is nothing to prove. Assume this is not the case. The function $\left(f+g\right)^{*}(s)^{\alpha-1}$ is non-increasing and right-continuous in $s$ on $(0,\mu(\RR))$. Therefore, by Proposition~\ref{P:maly} applied to $h=((f+g)^{*})^{\alpha-1}$ followed by the double use of the H\" older inequality~\eqref{E:holder} with $\varrho_{\frac{\alpha}{\alpha-1}}$ and $\varrho_{\alpha}$, we get
\begin{align*}
&\int_0^t\left(f+g\right)^{*}(s)^\alpha\,d s
=
\int_0^t\left(f+g\right)^{*}(s)^{\alpha-1}\left(f+g\right)^{*}(s)\,d s\\
&\leq
\int_0^t\left(f+g\right)^{*}(s)^{\alpha-1}\left(f^{*}(s)+g^{*}(s)\right)\,d s\\
& =
\int_0^t\left(f+g\right)^{*}(s)^{\alpha-1} f^{*}(s)\,d s+\int_0^t\left(f+g\right)^{*}(s)^{\alpha-1} g^{*}(s)\,d s\\
&\leq
\left(\int_0^t\left(f+g\right)^{*}(s)^{\alpha}\,d s\right)^{\frac{\alpha-1}{\alpha}}\left(\int_0^t f^{*}(s)^\alpha\,d s\right)^{\frac1{\alpha}}\\
&\qquad+
\left(\int_0^t\left(f+g\right)^{*}(s)^{\alpha}\,d s\right)^{\frac{\alpha-1}{\alpha}}\left(\int_0^t g^{*}(s)^\alpha\,d s\right)^{\frac1{\alpha}}\\
&=
\left(\int_0^t\left(f+g\right)^{*}(s)^{\alpha}\,d s\right)^{1-\frac{1}{\alpha}}\left(\left(\int_0^t f^{*}(s)^\alpha\,d s\right)^{\frac1{\alpha}}+\left(\int_0^t g^{*}(s)^\alpha\,d s\right)^{\frac1{\alpha}}\right).
\end{align*}
 Assume that for every $t\in(0,\mu(\RR))$ one has $\int_0^t\left(f+g\right)^{*}(s)^{\alpha}\,d s<\infty$. Since $f+g$ is not equal to the zero function $\mu$-almost everywhere on $\RR$, $\int_0^t\left(f+g\right)^{*}(s)^\alpha\,d s$ is positive regardless of~$t$. Hence, dividing, we get
\begin{align*}
\left(\int_0^t\left(f+g\right)^{*}(s)^{\alpha}\,d s\right)^{\frac{1}{\alpha}}\leq\left(\int_0^tf^{*}(s)^\alpha\,d s\right)^{\frac1{\alpha}}+\left(\int_0^tg^{*}(s)^\alpha\,d s\right)^{\frac1{\alpha}} \quad\text{for every $t\in(0,\mu(\RR))$.}
\end{align*}
Multiplying both sides by $t^{-\frac1{\alpha}}$, we arrive at pointwise estimate
\begin{align*}
\left((f+g)^\alpha\right)^{**}(t)^{\frac1{\alpha}}
\leq
\left(f^\alpha\right)^{**}(t)^{\frac1{\alpha}}
+\left(g^\alpha\right)^{**}(t)^{\frac1{\alpha}} \quad\text{for every $t\in(0,\mu(\RR))$.}
\end{align*}
Using (P2) and (P1) for $\overline{\varrho}$, we finally obtain
\begin{align*}
\varrho^{\aaa}(f+g)\le\varrho^{\aaa}(f) + \varrho^{\aaa}(g),
\end{align*}
establishing the statement. If, for some $t\in(0,\mu(\RR))$, one has $\int_0^t\left(f+g\right)^{*}(s)^{\alpha}\,d s=\infty$, then, by Remark~\ref{propsf**} and a change of variables, we get
\begin{align*}
    \int_0^t\left(f+g\right)^{*}(s)^{\alpha}\,d s
        & \le \int_0^t\left(f^*(\tfrac{s}{2})+g^*(\tfrac{s}{2})\right)^{\alpha}\,d s = 2 \int_0^{\frac{t}{2}}\left(f^*(\tau)+g^*(\tau)\right)^{\alpha}\,d \tau
            \\
        & \le 2 \int_0^{t}\left(f^*(\tau)+g^*(\tau)\right)^{\alpha}\,d \tau \le 2^{\alpha} \int_0^{t}\left(f^*(\tau)^{\alpha}+g^*(\tau)^{\alpha}\right)\,d \tau,
\end{align*}
so at least one of the terms $\int_0^tf^{*}(s)^\alpha\,d s$,
$\int_0^tg^{*}(s)^\alpha\,d s$ must be infinite. Assume with no loss of generality that $\int_0^tf^{*}(s)^\alpha\,d s=\infty$. Then, by monotonicity of $f^{*}$, we conclude that $(f^{\alpha})^{**}(s)=\infty$ for every $s\in(0,t)$. Consequently $\varrho^{\aaa}(f)=\infty$, the more so $\varrho^{\aaa}(f)+\varrho^{\aaa}(g)=\infty$, whence the assertion holds, again.  The proof is complete.
\end{proof}

One of the most important characteristics of any rearrangement-invariant structure is its fundamental function. We focus on it in the next theorem.

\begin{thm}[fundamental function]\label{FFekv}
Let $\varrho$ be a rearrangement-invariant norm over $(\RR,\mu)$ and let $\alpha\in(0,\infty)$. Let $\varphi^{\aaa}$ denote the fundamental function of $\varrho^{\aaa}$. Then
\begin{equation*}
    \varphi^{\aaa}\left(a\right) \approx a^{\frac{1}{\alpha}}\overline{\varrho}\left(t^{-\frac{1}{\alpha}}\chi_{\left(a,\mu(\RR)\right)}\left(t\right)\right) \quad\text{for $a\in \left(0,\tfrac{\mu(\RR)}{2}\right)$}
\end{equation*}
with constants of equivalence depending only on $\alpha$.
\end{thm}

\begin{proof}
Fix $a\in(0,\mu(\RR))$. Let $E\subset\RR$ such that $\mu(E)=a$. Then
\begin{equation*}
\left(\chi_{E}\right)^{**}(t)
=
\chi_{(0,a]}\left(t\right)+\frac{a}{t}\chi_{(a,\mu(\RR))}\left(t\right)\quad\text{for $t\in(0,\mu(\RR))$}.
\end{equation*}
 Thus,
 \begin{equation*}
\left(\chi_{E}\right)^{**}(t)^{\frac{1}{\alpha}}
=
\chi_{(0,a]}\left(t\right)+a^{\frac{1}{\alpha}}t^{-\frac{1}{\alpha}}\chi_{(a,\mu(\RR))}\left(t\right)\quad\text{for $t\in(0,\mu(\RR))$}.
\end{equation*}
Consequently,
\begin{align}
    \varphi^{\aaa}\left(a\right)
    = \overline{\varrho}\left(\left(\chi_{E}\right)^{**}(t)^{\frac{1}{\alpha}} \right)
    = \overline{\varrho}\left(\chi_{(0,a]}\left(t\right)+a^{\frac{1}{\alpha}}t^{-\frac{1}{\alpha}}\chi_{(a,\mu(\RR))}\left(t\right)\right).
    \label{E:FF}
\end{align}
Since all terms are non-negative, using (P2) and positive homogeneity for $\overline{\varrho}$, we immediately obtain the lower bound, namely
\begin{align}
    &\varphi^{\aaa}\left(a\right) \geq
    a^{\frac{1}{\alpha}}\overline{\varrho}\left(t^{-\frac{1}{\alpha}}\chi_{(a,\mu(\RR))}\left(t\right)\right). \label{E:varphi-lower-bound}
\end{align}
We will prove the upper bound. We get from~\eqref{E:FF} and (P1) for $\overline{\varrho}$
\begin{align}
    \varphi^{\aaa}\left(a\right)
        & \le \overline{\varrho}\left(\chi_{(0,a]}\right) + a^{\frac{1}{\alpha}}\overline{\varrho}\left(t^{-\frac{1}{\alpha}}\chi_{(a,\mu(\RR))}\left(t\right)\right).\label{E:varphi-upper}
\end{align}
Next we will show that for $a$ close to zero the first summand on the right in~\eqref{E:varphi-upper} is negligible. To this end, assume that $a\in \left(0,\tfrac{\mu(\RR)}{2}\right)$. Then, by positive homogeneity, (P2) and (P6) for $\overline{\varrho}$, we have
\begin{align*}
    \overline{\varrho}\left(t^{-\frac1{\alpha}}\chi_{(a,\mu(\RR))}(t)\right)
        & \ge \overline{\varrho}\left(t^{-\frac1{\alpha}}\chi_{(a,2a]}(t)\right)
        \ge (2a)^{-\frac{1}{\alpha}}\overline{\varrho}\left(\chi_{(a,2a]}\right)\nonumber
        = (2a)^{-\frac{1}{\alpha}}\overline{\varrho}\left(\chi_{(0,a]}\right).
\end{align*}
Plugging this into~\eqref{E:varphi-upper} and combining it with~\eqref{E:varphi-lower-bound}, we get
\begin{equation*}
    a^{\frac{1}{\alpha}}\overline{\varrho}\left(t^{-\frac{1}{\alpha}}\chi_{\left(a,\mu(\RR)\right)}\left(t\right)\right)
        \le \varphi^{\aaa}(a)
        \le \left(1+2^{\frac{1}{\alpha}}\right)
            a^{\frac{1}{\alpha}}\overline{\varrho}\left(t^{-\frac{1}{\alpha}}\chi_{\left(a,\mu(\RR)\right)}\left(t\right)\right),
\end{equation*}
establishing the claim.
\end{proof}

\begin{rem}
In Theorem~\ref{FFekv}, the case $\mu(\RR)=\infty$ is a possibility (then the equivalence holds for every $a\in(0,\infty)$). If $\mu(\RR)<\infty$, then, for $a$ close to $\mu(\mathscr{R})$, the result is clearly wrong. In such case, the best upper bound is given by~\eqref{E:varphi-upper}.
\end{rem}

\begin{thm}\label{T:P4}
Let $\varrho$ be a rearrangement-invariant norm over $(\RR,\mu)$ and let $\alpha\in(0,\infty)$. Then the following statements are equivalent:

\textup{(a)} $\varrho^{\aaa}$ satisfies \textup{(P4)},

\textup{(b)} either $\mu(\RR)<\infty$ or $\mu(\RR)=\infty$ and
\begin{equation}\label{E:P4}
    \overline{\varrho}\left(t^{-\frac{1}{\alpha}}\chi_{(a_0,\infty)}(t)\right)<\infty\quad\text{for some $a_0\in(0,\infty)$,}
\end{equation}

\textup{(c)} either $\mu(\RR)<\infty$ or $\mu(\RR)=\infty$ and
\begin{equation}\label{E:P4-all}
    \overline{\varrho}\left(t^{-\frac{1}{\alpha}}\chi_{(a,\infty)}(t)\right)<\infty\quad\text{for every $a\in(0,\infty)$.}
\end{equation}
\end{thm}

\begin{proof}
(a) $\Rightarrow$ (b) \quad Assume that $\mu(\RR)=\infty$ and~\eqref{E:P4} is false. Let $E$ be a $\mu$-measurable subset of $\RR$ such that $\mu(E)<\infty$. Let $\varphi^{\aaa}$ be the fundamental function of $\varrho^{\aaa}$.  Then, by Theorem~\ref{FFekv},
\begin{equation}\label{E:fundamental-estimate}
   \varrho^{\aaa}(\chi_{E}) = \varphi^{\aaa}(\mu(E))\approx \mu(E)^{\frac{1}{\alpha}}\overline{\varrho}\left(t^{-\frac{1}{\alpha}}\chi_{(\mu(E),\infty)}(t)\right),
\end{equation}
which is infinite thanks to the negation of~\eqref{E:P4}. Consequently, $\varrho^{\aaa}$ does not satisfy (P4).

\medskip
(b) $\Rightarrow$ (c) \quad If $\mu(\RR)<\infty$, the implication is trivial. Assume that $\mu(\RR)=\infty$ and $a\in(0,\infty)$. If $a\in[a_0,\infty)$, then, thanks to (P2) for $\overline{\varrho}$,~\eqref{E:P4-all} follows immediately from~\eqref{E:P4}. Let $a\in(0,a_0)$. By (P1) and (P2) for $\overline{\varrho}$,
\begin{align*}
    \overline{\varrho}\left(t^{-\frac{1}{\alpha}}\chi_{(a,\infty)}(t)\right)
        & \le \left(\overline{\varrho}\left(t^{-\frac{1}{\alpha}}\chi_{(a,a_0)}(t)\right) + \overline{\varrho}\left(t^{-\frac{1}{\alpha}}\chi_{(a_0,\infty)}(t)\right)\right)
            \\
        & \le a^{-\frac{1}{\alpha}}\overline{\varrho}\left(\chi_{(a,a_0)}(t)\right) + \overline{\varrho}\left(t^{-\frac{1}{\alpha}}\chi_{(a_0,\infty)}(t)\right),
\end{align*}
which is finite by (P4) for $\overline{\varrho}$ and~\eqref{E:P4}. This establishes (c).

(c) $\Rightarrow$ (a) \quad
 Assume first that $\mu(\RR)=\infty$ and that~\eqref{E:P4-all} holds. Let $E$ be a $\mu$-measurable subset of $\RR$ such that $\mu(E)<\infty$. We know from~\eqref{E:fundamental-estimate} that
\begin{align*}
    \varrho^{\aaa}(\chi_{E}) \approx \mu(E)^{\frac{1}{\alpha}}\overline{\varrho}\left(t^{-\frac{1}{\alpha}}\chi_{(\mu(E),\infty)}(t)\right),
\end{align*}
which in turn is finite by~\eqref{E:P4-all}, so the claim follows. Now assume that $\mu(\RR)<\infty$. Then, by the very definition of $\varrho^{\aaa}$, (P2) for $\varrho^{\aaa}$ and (P4) for $\overline{\varrho}$, we have
\begin{equation*}
   \varrho^{\aaa}(\chi_{E}) \le \varrho^{\aaa}(\chi_{\RR}) = \overline{\varrho}\left(\chi_{(0,\mu(\RR))}\right) < \infty,
\end{equation*}
proving the claim again. The proof is complete.
\end{proof}

\begin{rem}
Let $\varrho$ be a rearrangement-invariant norm over $(\RR,\mu)$ and let $\alpha\in(0,\infty)$. Assume that either $\mu(\RR)<\infty$ or $\mu(\RR)=\infty$ and~\eqref{E:P4} holds. Then, by Theorems~\ref{quasinorm} and~\ref{T:P4}, $\varrho^{\aaa}$ is an r.i.~quasinorm.
\end{rem}

\begin{rem}\label{R:pess}
Let $\varrho$ be a rearrangement-invariant norm over $(\RR,\mu)$, $X=X_{\varrho}$ and $\alpha\in(0,\infty)$. Then one always has $0\in X^{\aaa}$, where $0$ stands for the function which is equal to zero $\mu$-a.e.~on $\RR$. In exceptional cases it may happen that $0$ is the only element of $X^{\aaa}$. In such cases we shall say that $X^{\aaa}$ is \textit{trivial}. Clearly, $X^{\aaa}$ is non-trivial if and only if $\varrho^{\aaa}$ satisfies (P4). It thus follows from Theorem~\ref{T:P4} that $X^{\aaa}$ is non-trivial if and only if either $\mu(\RR)<\infty$ or $\mu(\RR)=\infty$ and~\eqref{E:P4} holds.
\end{rem}

We shall now apply the results obtained to Lebesgue and Lorentz spaces.

\begin{proposition}[Lebesgue spaces]\label{P:triviality-lebesgue}
Let $p\in[1,\infty]$ and $\alpha\in(0,\infty)$. Then $(L^p)^{\aaa}$ is non-trivial if and only if either $\mu(\mathscr{R})<\infty$ or $\alpha\in(0,p)$.
\end{proposition}

\begin{proof}
According to Remark~\ref{R:pess}, $(L^p)^{\aaa}$ is non-trivial if and only if either $\mu(\mathscr{R})<\infty$ or $\mu(\mathscr{R})=\infty$ and there exists an $a\in(0,\infty)$ such that $\|t^{-\frac{1}{\alpha}}\chi_{(a,\infty)}\|_{L^p(0,\infty)}<\infty.$ This, in turn, is true if and only if $\alpha\in(0,p)$.
\end{proof}

\begin{proposition}[Lorentz spaces]\label{P:triviality-lorentz}
Let either $p=q=1$ or $p=q=\infty$ or $p\in(1,\infty)$, $q\in[1,\infty]$. Let $\alpha\in(0,\infty)$. Then $\left(L^{p,q}\right)^{\aaa}$ is non-trivial if and only if one of the following conditions holds:
\begin{itemize}
\item $\mu(\RR)<\infty$,
\item $\alpha\in(0,p)$ and $q\in[1,\infty)$,
\item $\alpha\in(0,p]$ and $q=\infty$.
\end{itemize}
\end{proposition}

\begin{proof}
By Remark~\ref{R:pess} once again, $\left(L^{p,q}\right)^{\aaa}$ is non-trivial if and only if either $\mu(\mathscr{R})<\infty$ or $\mu(\mathscr{R})=\infty$ and there exists an $a\in(0,\infty)$ such that $\|t^{-\frac{1}{\alpha}}\chi_{(a,\infty)}(t)\|_{L^{p,q}(0,\infty)}<\infty$. Assume that $\mu(\mathscr{R})=\infty$ and $q\in[1,\infty)$. Then
\begin{align*}
    \left\|t^{-\frac1{\alpha}}\chi_{(a,\infty)}(t)\right\|_{L^{p,q}(0,\infty)}^{q} = \int_0^\infty \left(t^{\frac1p}(t+a)^{-\frac1{\alpha}}\right)^q\frac{dt}t,
\end{align*}
which converges if and only if $\alpha\in(0,p)$. Now assume that $\mu(\mathscr{R})=\infty$ and $q=\infty$. Then
$$
\left\|t^{-\frac1{\alpha}}\chi_{(a,\infty)}(t)\right\|_{L^{p,\infty}(0,\infty)} = \operatorname{ess\,sup}_{t\in (0,\infty)} t^{\frac1p}(t+a)^{-\frac1{\alpha}},
$$
which is finite if and only if $\alpha\in(0,p]$. The proof is complete.
\end{proof}

\begin{thm}\label{XaBFS}
Let $\varrho$ be a rearrangement-invariant norm over $(\RR,\mu)$ and let $\alpha\in[1,\infty)$. Then $\varrho^{\aaa}$ is an r.i.~norm if and only if either $\mu(\RR)<\infty$ or $\mu(\RR)=\infty$ and~\eqref{E:P4} holds.
\end{thm}

\begin{proof}
Assume that either $\mu(\RR)<\infty$ or $\mu(\RR)=\infty$ and~\eqref{E:P4} holds. Then property (P1) for $\varrho^{\aaa}$ follows from Theorem~\ref{norm}, properties (P2), (P3) and (P6) from Theorem~\ref{quasinorm}
and property (P4) from Theorem~\ref{T:P4}. It remains to verify property (P5). To this end, let $E\subset \RR$ be a $\mu$-measurable set such that $\mu(E)<\infty$ and let $f\in\M_+(\RR,\mu)$. Then, by (P5) for $\varrho$
and Proposition~\ref{XaembX},
\begin{equation*}
    \int_E f\,d\mu \le C_E \varrho(f) \le C_E \varrho^{\aaa}(f).
\end{equation*}

Conversely, assume that $\varrho^{\aaa}$ is an r.i.~norm. Then it satisfies (P4). Hence, the conclusion follows from~Theorem~\ref{T:P4}. The proof is complete.
\end{proof}

We shall return to the question of necessity of the assumption $\alpha\in[1,\infty)$ in Theorem~\ref{XaBFS} in the subsequent section in Remark~\ref{R:alpha}.

\section{Embeddings}\label{S:embeddings}

Our next aim is to study embedding relations between the structures of the form $X^{\aaa}$, where $X=X_{\varrho}$ for some given functional $\varrho$ and $X^{\aaa}=X_{\varrho^{\aaa}}$. This fact will entail a slight change of our point of view. So far our main focus was concentrated on norms whereas from this section on the main object of our research will be spaces. In correspondence with this intention we shall, from this section on, mostly work with $X$, $X^{\aaa}$ and $\overline{X}$ rather than with $\varrho$, $\varrho^{\aaa}$ and $\overline{\varrho}$, respectively.

We begin with a~simple nesting property.

\begin{proposition}
Let $X$ be a rearrangement-invariant space over $(\RR,\mu)$, $\alpha_1,\alpha_2\in(0,\infty)$ and $\alpha_1\leq\alpha_2$. Then $\|f\|_{X^{\langle\alpha_1\rangle}} \le \|f\|_{{X^{\langle\alpha_2\rangle}}}$ for every $f\in\M(\RR,\mu)$.
\end{proposition}

\begin{proof}
The function $t\mapsto t^{\frac{\alpha_2}{\alpha_1}}$ is convex on $(0,\mu(\RR))$. Hence, by Jensen's inequality and (P2) for~$\|\cdot\|_{\overline{X}(0,\mu(\RR))}$ restricted to non-negative functions, we have
\begin{align*}
\|f\|_{X^{\langle\alpha_1\rangle}}
&=
\left\|\left(\frac1t\int_0^t f^{*}(s)^{\alpha_1}\,ds\right)^{\frac{1}{\alpha_1}}\right\|_{\overline{X}(0,\mu(\RR))}
\leq
\left\|\left(\frac1t\int_0^t f^{*}(s)^{\alpha_2}\,ds\right)^{\frac{1}{\alpha_2}}\right\|_{\overline{X}(0,\mu(\RR))}
=
\|f\|_{{X^{\langle\alpha_2\rangle}}}.
\end{align*}
\end{proof}

We next point out certain stability property with respect to continuous embeddings.

\begin{proposition}
Let $X,Y$ be rearrangement-invariant spaces over $(\RR,\mu)$ such that $\|f\|_{X} \le \|f\|_{Y}$ for every $f\in\M(\RR,\mu)$ and let $\alpha\in(0,\infty)$. Then $\|f\|_{X^{\langle\alpha\rangle}} \le \|f\|_{{Y^{\langle\alpha\rangle}}}$ for every $f\in\M(\RR,\mu)$.
\end{proposition}

\begin{proof}
The assertion follows immediately from the definitions.
\end{proof}

\begin{rem}
It follows from Remark~\ref{XaembX} that $X^{\aaa}\hookrightarrow X$ for every r.i.~space $X$ and every $\alpha\in(0,\infty)$. We shall now investigate the question, under which additional
conditions the converse embedding, namely $X\hookrightarrow X^{\aaa}$, holds. In view of Remark~\ref{XaembX}, in such case we in fact have $X=X^{\aaa}$. The question of
characterizing such situations is of a~considerable interest in applications, see e.g.~\cite[proof of Theorem~3.1]{CPS-frostman-2}.
\end{rem}

We start by pointing out that an~interesting and~useful characterization is available, expressed in terms of boundedness of the Hardy averaging operator on an appropriate space involving functions defined on an interval. For $h\in\M_+(0,\mu(\RR))$, we define the averaged function, $Ah$, by
\begin{equation*}
    Ah(t)=\frac{1}{t}\int_{0}^{t} h(s) \,d s \quad \text{for $t\in (0,\mu(\RR))$.}
\end{equation*}
 By $A$ we denote the operator which associates every admissible function $h$ with $Ah$. This operator is then called the \textit{Hardy averaging operator}.

\begin{thm}\label{T:boundedness-of-M}
Let $X$ be a rearrangement-invariant space over $(\RR,\mu)$ and let $\alpha\in(0,\infty)$. Then $X\hookrightarrow X^{\aaa}$ if and only if there exists a positive constant $\kappa$ such that
\begin{equation}\label{E:boundedness-of-M}
    \|Ah\|_{\overline{X}^{\left\{\frac{1}{\alpha}\right\}}(0,\mu(\RR))} \le \kappa \|h\|_{\overline{X}^{\left\{\frac{1}{\alpha}\right\}}(0,\mu(\RR))}
\end{equation}
for every non-increasing $h\in\M_+(0,\mu(\RR))$. Moreover, if $\kappa$ is the optimal constant in~\eqref{E:boundedness-of-M} and $c$ is the norm of the embedding $X\hookrightarrow X^{\aaa}$, then $\kappa= c^{\alpha}$.
\end{thm}

\begin{proof}
The embedding $X\hookrightarrow X^{\aaa}$ holds if and only if there exists a positive constant $c$ such that
\begin{equation*}
    \|f\|_{X^{\aaa}} \le c \|f\|_{X} \quad\text{for every $f\in\M(\RR,\mu)$},
\end{equation*}
that is,
\begin{equation*}
    \left\|\left(\frac{1}{t}\int_{0}^{t}f^{*}(s)^{\alpha}\,ds\right)^{\frac{1}{\alpha}}\right\|_{\overline{X}(0,\mu(\RR))} \le c \|f^{*}\|_{\overline{X}(0,\mu(\RR))} \quad\text{for every $f\in\M(\RR,\mu)$},
\end{equation*}
A simple substitution shows that the last inequality is equivalent to saying that
\begin{equation*}
    \left\|\left(\frac{1}{t}\int_{0}^{t}g^{*}(s)\,ds\right)^{\frac{1}{\alpha}}\right\|_{\overline{X}(0,\mu(\RR))} \le c \|(g^{*})^{\frac{1}{\alpha}}\|_{\overline{X}(0,\mu(\RR))} \quad\text{for every $g\in\M(\RR,\mu)$},
\end{equation*}
Raising this to $\alpha$ we get
\begin{equation*}
    \|Ag^{*}\|_{\overline{X}^{\left\{\frac{1}{\alpha}\right\}}(0,\mu(\RR))} \le c^{\alpha} \|g^{*}\|_{\overline{X}^{\left\{\frac{1}{\alpha}\right\}}(0,\mu(\RR))}\quad\text{for every $g\in\M(\RR,\mu)$,}
\end{equation*}
and it just remains to realize that, given a non-increasing $h\in\M_+(0,\mu(\RR))$, we can always find a function $g\in\M(\RR,\mu)$ such that $g^{*}=h$ thanks to the fact that the $(\RR,\mu)$ is non-atomic (see e.g.~\cite[Corollary~7.8]{BS} or~\cite[Lemma~2.2]{MO}).
\end{proof}

\begin{thm}\label{T:upper-norm}
Let $X$ be a rearrangement-invariant space over $(\RR,\mu)$ and let $\alpha\in(0,1)$. Then
\begin{equation*}
    \|f\|_{X^{\aaa}} \le \left(\frac{1}{1-\alpha}\right)^{\frac{1}{\alpha}}\|f\|_{X}\quad\text{for every $f\in\M(\RR,\mu)$.}
\end{equation*}
\end{thm}

\begin{proof}
Assume that $f\in\M(\RR,\mu)$ and fix $t\in(0,\mu(\RR))$. The classical Hardy inequality asserts that
\begin{equation*}
    \int_{0}^{t}\left(\frac{1}{s}\int_{0}^{s}h(\tau)\,d\tau\right)^{\frac{1}{\alpha}}\,ds
        \le \left(\frac{1}{1-\alpha}\right)^{\frac{1}{\alpha}} \int_{0}^{t}h(s)^{\frac{1}{\alpha}}\,ds
            \quad\text{for every $h\in\M_+(0,\mu(\RR))$.}
\end{equation*}
Applying this, in particular, to $h=(|f|^{\alpha})^{*}$, we obtain
\begin{equation*}
    \int_{0}^{t}\left(\frac{1}{s}\int_{0}^{s}f^*(\tau)^{\alpha}\,d\tau\right)^{\frac{1}{\alpha}}\,ds
        \le \left(\frac{1}{1-\alpha}\right)^{\frac{1}{\alpha}} \int_{0}^{t}f^*(s)\,ds.
\end{equation*}
Both the functions $s\mapsto \left(\frac{1}{s}\int_{0}^{s}f^*(\tau)^{\alpha}\,d\tau\right)^{\frac{1}{\alpha}}$ and $s\mapsto f^*(s)$ are obviously non-increasing. Thus, we get, by the Hardy--Littlewood--P\'olya principle, \begin{equation*}
    \left\|\left(\frac{1}{s}\int_{0}^{s}f^*(\tau)^{\alpha}\,d\tau\right)^{\frac{1}{\alpha}}\right\|_{\overline{X}(0,\mu(\RR))}
        \le \left(\frac{1}{1-\alpha}\right)^{\frac{1}{\alpha}} \|f^{*}\|_{\overline{X}(0,\mu(\RR))},
\end{equation*}
and the assertion follows.
\end{proof}

\begin{rem}\label{R:alpha}
If $X$ is a rearrangement-invariant space over $(\RR,\mu)$ and $\alpha\in(0,1)$, then $X=X^{\aaa}$. Indeed, by Proposition~\ref{XaembX} and Theorem~\ref{T:upper-norm}, we have
\begin{equation*}
    \|f\|_{X} \le\|f\|_{X^{\aaa}} \le  \left(\frac{1}{1-\alpha}\right)^{\frac{1}{\alpha}}\|f\|_{X} \quad\text{for every $f\in\M(\RR,\mu)$.}
\end{equation*}
In particular, $X^{\aaa}$ is equivalent to a rearrangement-invariant space. For $\alpha\in[1,\infty)$, this is not necessarily true. Note that this fact complements the result of Theorem~\ref{XaBFS}.
\end{rem}

We shall now give several examples describing the action of the operation $X\mapsto X^{\aaa}$ on some important scales of function spaces.

\begin{example}[Lebesgue spaces]
Let $p\in[1,\infty]$ and $\alpha\in(0,\infty)$. Then
\begin{equation*}
(L\sp p)^{\aaa}=L\sp p\quad\text{if and only if $\alpha\in(0,p)$.}
\end{equation*}
In particular,
\begin{equation*}
    (L\sp{\infty})^{\aaa}=L\sp{\infty}\quad\text{for any $\alpha\in(0,\infty)$.}
\end{equation*}
Indeed, if $X=L\sp{p}$ and $\alpha\in(0,\infty)$, then a simple calculation shows that
\begin{equation*}
\overline{X}\sp{\left\{\frac{1}{\alpha}\right\}}(0,\mu(\RR)) = L^{\frac{p}{\alpha}}(0,\mu(\RR)).
\end{equation*}
Hence the condition~\eqref{E:boundedness-of-M} is satisfied if and only if the Hardy averaging operator is bounded from $L^{\frac{p}{\alpha}}(0,\mu(\RR))$ into itself, which is known to be true if and only if $\frac{p}{\alpha}>1$. The claim thus follows from Theorem~\ref{T:boundedness-of-M} and Proposition~\ref{XaembX}.
\end{example}

We shall now turn our attention to the cases when $\alpha\in[p,\infty)$, in which, interestingly, the situation is considerably different. To begin, note that, unlikely in the subcritical case $\alpha\in(0,p)$, in both the critical case $\alpha=p$ and the supercritical case $\alpha\in(p,\infty)$ we have to assume that $\mu(\RR)$ is finite. This is caused by the fact, which follows from Proposition~\ref{P:triviality-lebesgue}, that if $\mu(\RR)=\infty$ and $\alpha\in[p,\infty)$, then $(L^{p})^{\aaa}$ is trivial.

\begin{example}[Lebesgue spaces, critical and supercritical cases]\label{EX:lebesgue-limiting}
Let $\mu(\RR)<\infty$, $p\in[1,\infty)$ and $\alpha\in[p,\infty)$. Then
\begin{equation}\label{E:lebesgue-limiting}
    (L\sp p)^{\left\langle  \alpha \right\rangle}
        =
        \begin{cases}
            L\sp p(\log L)^\frac1p &\text{if $\alpha=p$},
                \\
            L^{\alpha} &\text{if $\alpha\in(p,\infty)$}.
        \end{cases}
\end{equation}
To prove the first claim in~\eqref{E:lebesgue-limiting}, it suffices, in view of the above-mentioned relations, to show that the functionals $f\mapsto\|f\|_{(L^{p})^{\aaa}}$ and $f\mapsto\|f\|_{L^{p,p;\frac{1}{p}}}$ are comparable on $\M(\RR,\mu)$. To this end, we note that, by the Fubini theorem, we have, for every $f\in\M(\RR,\mu)$,
\begin{align*}
    \|f\|_{(L^{p})^{\aaa}}
    & =
    \left\|\left(\left(|f|\sp p\right)\sp{**}\right)\sp{\frac1p}\right\|_{L^{p}(0,\mu(\RR))}
    =
    \left(\int_0\sp{\mu(\RR)}\frac1t\int_0\sp tf\sp*(s)\sp p\,ds\,dt\right)\sp{\frac1p}
            \\
        & =
    \left(\int_0\sp{\mu(\RR)}f\sp*(s)\sp p\log\tfrac{\mu(\RR)}s\,ds\right)\sp{\frac1p}.
\end{align*}
Thus,
\begin{equation*}
    \|f\|_{(L^{p})^{\ppp}} \le \|f\|_{L^{p,p;\frac{1}{p}}}.
\end{equation*}
On the other hand, by changing variables and using the monotonicity of $f^{*}$, we get
\begin{align*}
    \|f\|_{(L^{p})^{\ppp}}
    & \ge
    \left(\int_0\sp{\frac{\mu(\RR)}{e}}f\sp*(s)\sp p\log\tfrac{\mu(\RR)}s\,ds\right)\sp{\frac1p}
    \ge
    \left(\frac{1}{e}\int_0\sp{\mu(\RR)}f\sp*\left(\tfrac{s}{e}\right)\sp p\log\tfrac{e\mu(\RR)}s\,ds\right)\sp{\frac1p}
        \\
    & \ge
    \left(\frac{1}{e}\int_0\sp{\mu(\RR)}f\sp*\left(s\right)\sp p\log\tfrac{e\mu(\RR)}s\,ds\right)\sp{\frac1p} = e^{-\frac{1}{p}} \|f\|_{L^{p,p;\frac{1}{p}}}.
\end{align*}
Altogether,
\begin{equation*}
    e^{-\frac{1}{p}} \|f\|_{L^{p,p;\frac{1}{p}}}\le  \|f\|_{(L^{p})^{\ppp}}  \le \|f\|_{L^{p,p;\frac{1}{p}}},
\end{equation*}
which yields the desired relation. Now assume that $\alpha\in(p,\infty)$. In order to prove the second claim in~\eqref{E:lebesgue-limiting}, we now have to verify that the functionals $f\mapsto\|f\|_{(L^{p})^{\aaa}}$ and $f\mapsto\|f\|_{L^{\alpha}}$ are comparable on $\M(\RR,\mu)$. One has
\begin{align*}
    \|f\|_{(L^{p})^{\aaa}}
    & =
    \left(\int_0\sp{\mu(\RR)}\left(\frac1t\int_0\sp tf\sp*(s)\sp {\alpha}\,ds\right)^{\frac{p}{\alpha}}\,dt\right)\sp{\frac1p}
        \\
    & \le
    \left(\int_0\sp {\mu(\RR)}f\sp*(s)\sp {\alpha}\,ds\right)^{\frac{1}{\alpha}}
    \left(\int_0\sp{\mu(\RR)}t^{-\frac{p}{\alpha}}\, dt\right)\sp{\frac1p}
        \\
    & =
    \left(\frac{\alpha}{\alpha-p}\right)^{\frac{1}{p}}\mu(\RR)^{\frac{1}{p}-\frac{1}{\alpha}}
    \left(\int_0\sp {\mu(\RR)}f\sp*(s)\sp {\alpha}\,ds\right)^{\frac{1}{\alpha}}.
\end{align*}
Conversely, using~\eqref{E:rozpulil-pess}, we get
\begin{align*}
    \|f\|_{(L^{p})^{\aaa}}
    & \ge
    \left(\int_{\frac{\mu(\RR)}{2}}\sp{\mu(\RR)}\left(\frac1t\int_0\sp tf\sp*(s)\sp {\alpha}\,ds\right)^{\frac{p}{\alpha}}\,dt\right)\sp{\frac1p}
        \\
    & \ge
    \left(\int_0\sp {\frac{\mu(\RR)}{2}}f\sp*(s)\sp {\alpha}\,ds\right)^{\frac{1}{\alpha}}
    \left(\int_{\frac{\mu(\RR)}{2}}\sp{\mu(\RR)}t^{-\frac{p}{\alpha}}\, dt\right)\sp{\frac1p}
        \\
    & \ge
    \left(\frac{\alpha}{\alpha-p}\right)^{\frac{1}{p}}
    \left(1-2^{\frac{p-\alpha}{\alpha}}\right)^{\frac{1}{p}}\mu(\RR)^{\frac{1}{p}-\frac{1}{\alpha}}
   \left(\int_0\sp {\frac{\mu(\RR)}{2}}f\sp*(s)\sp {\alpha}\,ds\right)^{\frac{1}{\alpha}}
        \\
    & \ge
    \left(\frac{\alpha}{\alpha-p}\right)^{\frac{1}{p}}2^{-\frac{1}{\alpha}}
    \left(1-2^{\frac{p-\alpha}{\alpha}}\right)^{\frac{1}{p}}\mu(\RR)^{\frac{1}{p}-\frac{1}{\alpha}}
    \left(\int_0\sp {\mu(\RR)}f\sp*(s)\sp {\alpha}\,ds\right)^{\frac{1}{\alpha}},
\end{align*}
and the second claim in~\eqref{E:lebesgue-limiting} follows on combining the last two estimates.
\end{example}

\begin{example}[Lorentz spaces]\label{EX:lorentz}
Let either $p=q=1$ or $p=q=\infty$ or $p\in(1,\infty)$, $q\in[1,\infty]$. Let $\alpha\in(0,\infty)$. If $\mu(\RR)<\infty$, then
\begin{equation}\label{E:lorentz-examples-1}
    (L\sp {p,q})^{\aaa}=
        \begin{cases}
            L\sp{p,q} &\text{if $\alpha\in(0,p)$,}
                \\
            Y_{p,q} &\text{if $\alpha=p$,}
                \\
            L^{\alpha} &\text{if $\alpha\in(p,\infty)$,}
        \end{cases}
\end{equation}
where $Y_{p,q}$ is the collection of all $f\in\M_0(\RR,\mu)$ such that $\|f\|_{Y_{p,q}}<\infty$, where
\begin{equation}\label{E:Y-pq}
   \|f\|_{Y_{p,q}}=
        \begin{cases}
            \left(\int_0\sp{\mu(\RR)}\left(\int_0\sp t f\sp*(s)\sp p\,ds\right)\sp{\frac qp}\frac{dt}t\right)\sp{\frac1q}
            &\text{if $q\in[1,\infty)$,}
                \\
            \|f\|_{L^{p}}&\text{if $q=\infty$}
        \end{cases}
\end{equation}
for $f\in\M(\RR,\mu)$. If $\mu(\RR)=\infty$, then
\begin{equation}\label{E:lorentz-examples-2}
    (L\sp {p,q})^{\aaa}=
        \begin{cases}
            L\sp{p,q} &\text{if $\alpha\in(0,p)$,}
                \\
            L^{p} &\text{if $\alpha=p$ and $q=\infty$,}
                \\
             \{0\} &\text{if either $\alpha=p$ and $q\in[1,\infty)$ or $\alpha\in(p,\infty)$.}
          \end{cases}
\end{equation}
Indeed, if $X=L\sp{p,q}$ and $\alpha\in(0,\infty)$, then a simple calculation shows that
\begin{equation*}
\overline{X}\sp{\left\{\frac{1}{\alpha}\right\}} (0,\mu(\RR)) = L^{\frac{p}{\alpha},\frac{q}{\alpha}}(0,\mu(\RR)).
\end{equation*}
Assume first that $\alpha\in(0,p)$. Then it is known ( cf.~e.g.~\cite[Chapter~4, Lemma~4.5]{BS}) that there exists a constant $\kappa$ such that
\begin{equation*}
    \|Ah\|_{L^{\frac{p}{\alpha},\frac{q}{\alpha}}(0,\mu(\RR))} \le \kappa \|h\|_{L^{\frac{p}{\alpha},\frac{q}{\alpha}}(0,\mu(\RR))}
\end{equation*}
for every non-increasing $h\in\M_+(0,\mu(\RR))$. This is true regardless of the finiteness of $\mu(\RR)$, hence the first claims in both~\eqref{E:lorentz-examples-1} and~\eqref{E:lorentz-examples-2} follow from Theorem~\ref{T:boundedness-of-M} and Proposition~\ref{XaembX}. If $\alpha=p$, $q\in[1,\infty)$ and $\mu(\RR)<\infty$, then the second claim in~\eqref{E:lorentz-examples-1} follows straightforward from the definitions. If $\alpha=p$ and $q=\infty$, then we have, for every $f\in\M(\RR,\mu)$,
\begin{align*}
    \|f\|_{(L^{p,\infty})^{\ppp}}
    &=
    \sup_{t\in(0,\mu(\RR))}t\sp{\frac1p}\left(\frac1t\int_0\sp t f\sp*(s)\sp p\,ds\right)\sp{\frac 1p}
    = \sup_{t\in(0,\mu(\RR))}\left(\int_0\sp t f\sp*(s)\sp p\,ds\right)\sp{\frac 1p}
        \\
    &=
    \left(\int_0\sp {\mu(\RR)} f\sp*(s)\sp p\,ds\right)\sp{\frac 1p} = \|f\|_{L^{p}},
\end{align*}
again regardless of the finiteness of $\mu(\RR)$. We have thus verified the second claims in~\eqref{E:lorentz-examples-1} and~\eqref{E:lorentz-examples-2}. Assume that $\mu(\RR)<\infty$ and $\alpha\in(p,\infty)$. Then
\begin{align*}
    \|f\|_{(L\sp {p,q})^{\aaa}}
        & = \left\|t^{\frac{1}{p}-\frac{1}{\alpha}-\frac{1}{q}}\left(\int_{0}^{t}f^{*}(s)^{\alpha}\,ds\right)^{\frac{1}{\alpha}}\right\|_{L^{q}(0,\mu(\RR))}
            \\
        & \le \left(\int_{0}^{\mu(\RR)}f^{*}(s)^{\alpha}\,ds\right)^{\frac{1}{\alpha}} \left\|t^{\frac{1}{p}-\frac{1}{\alpha}-\frac{1}{q}}\right\|_{L^{q}(0,\mu(\RR))}
        = C_{p,q,\alpha} \|f\|_{L^{\alpha}(0,\mu(\RR))}
\end{align*}
and, at the same time,
\begin{align*}
    \|f\|_{(L\sp {p,q})^{\aaa}}
        & \ge \left\|t^{\frac{1}{p}-\frac{1}{\alpha}-\frac{1}{q}}\left(\int_{0}^{t}f^{*}(s)^{\alpha}\,ds\right)^{\frac{1}{\alpha}}\right\|_{L^{q}(\frac{\mu(\RR)}{2},\mu(\RR))}
            \\
        & \ge \left(\int_{0}^{\frac{\mu(\RR)}{2}}f^{*}(s)^{\alpha}\,ds\right)^{\frac{1}{\alpha}} \left\|t^{\frac{1}{p}-\frac{1}{\alpha}-\frac{1}{q}}\chi_{(\frac{\mu(\RR)}{2},\mu(\RR))}(t)\right\|_{L^{q}(0,\mu(\RR))}
        \ge c_{p,q,\alpha} \|f\|_{L^{\alpha}(0,\mu(\RR))}
\end{align*}
in which $c_{p,q,\alpha}$ and $C_{p,q,\alpha}$ are finite positive constants depending only on the indicated parameters. This shows the third claim in~\eqref{E:lorentz-examples-1}. Finally, the third claim in~\eqref{E:lorentz-examples-2} follows from Proposition~\ref{P:triviality-lorentz}.

Note that when $q=p$, we recover the information from Example~\ref{EX:lebesgue-limiting}.
\end{example}

The spaces $Y_{p,q}$ that surfaced in the course of Example~\ref{EX:lorentz} are of independent interest. It turns out that for a fixed $p\in[1,\infty)$, the family $\{Y_{p,q}: q\in[p,\infty]\}$ forms an~important scale of function spaces that are not directly comparable to customary function spaces. Moreover, it represents a certain bridge from $L\sp p$ to $L\sp p(\log L)^\frac1p$. This issue will be discussed in the next section in detail.

\section{Two ways of bridging the gap}\label{S:bridge}

We shall assume throughout this section that $\mu(\RR)<\infty$. It follows from the results in the previous section that, for any $p\in[1,\infty)$, the spaces $L\sp p$ and $L\sp p(\log L)^\frac1p$ can be bridged by the scale of spaces $Y_{p,q}$ defined by~\eqref{E:Y-pq}, as $q$ is ranging from $p$ to $\infty$. The endpoint space of this scale obtained by putting $q=p$ gives $L\sp p(\log L)^\frac1p$, while, on the other side, setting $q=\infty$ we get the opposite endpoint, namely $L\sp p$. Since there is another natural bridge between these two spaces represented by the scale of Lorentz--Zygmund spaces $\{L^{p,p;\alpha}: \alpha\in[0,\frac{1}{p}]\}$ in which the
endpoint $\alpha=0$ corresponds to $L\sp p$ and the endpoint $\alpha=\frac1p$ corresponds to $L\sp p(\log L)^\frac1p$, it is of interest to investigate which relations hold between the spaces taking part in the two scales. Consider the diagram:

\begin{center}
\scalebox{.3}{%
\includegraphics{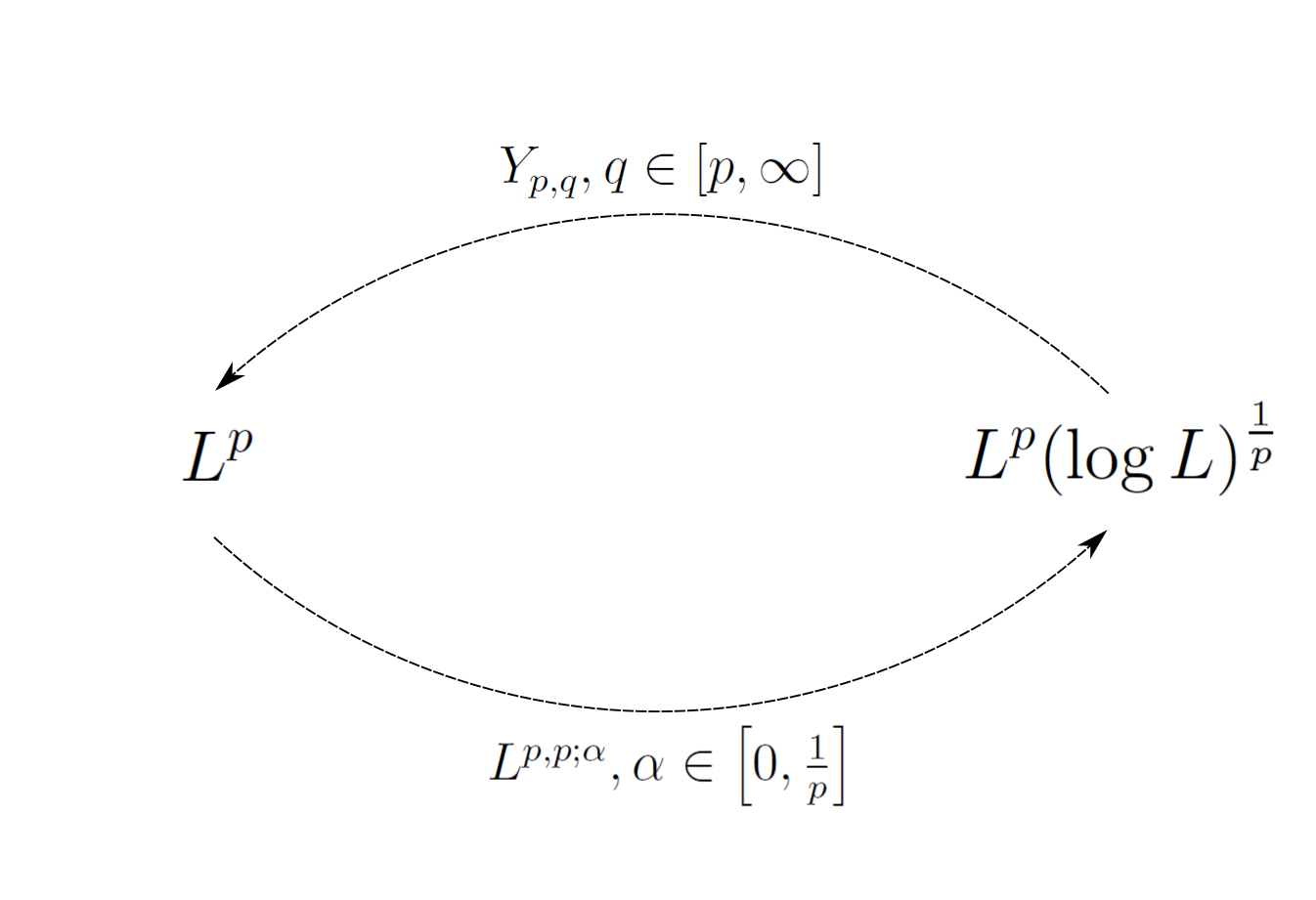}}
\end{center}

The answer to this question is provided by the following result which gives us a precise image of the positioning of a space $Y_{p,q}$ with respect to the scale $\{L^{p,p;\alpha}: \alpha\in[0,\frac{1}{p}]\}$.

\begin{thm}
Assume that $\mu(\RR)<\infty$. Let $p\in[1,\infty)$, $q\in (p,\infty)$ and $\alpha\in(0,\frac1q)$. Then
\begin{equation}\label{E:sandwich}
    L^{p,p;\frac{1}{q}} \hookrightarrow Y_{p,q}\hookrightarrow L^{p,p;\alpha},
\end{equation}
and neither of the embeddings in~\eqref{E:sandwich} can be reversed.
\end{thm}

\begin{proof}
To simplify the calculations, let us assume, without loss of generality, that $\mu(\RR)=1$. It follows straightaway from definitions that, given $\alpha\in(0,\infty)$, the embedding $L^{p,p;\alpha} \hookrightarrow Y_{p,q}$ holds if and only if there exists $C>0$ such that the inequality
\begin{align}\label{EQ:>C}
\left(\int_0^1\left(\int_0^tf^*(s)^p\,ds\right)^{\frac{q}{p}}\frac{dt}{t}\right)^{\frac{1}{q}} \le C \left(\int_0^1f^*(t)^p\left(\log \frac{e}{t}\right)^{\alpha p}\,dt\right)^\frac1p
\end{align}
is satisfied for every $f\in\M(\RR,\mu)$.
As we already know (see the proof of Theorem~\ref{T:boundedness-of-M}), given such $f\in\M(\RR,\mu)$, there always exists $g\in\M(\RR,\mu)$ satisfying $(f^*)^p=g^*$. Therefore,~\eqref{EQ:>C} holds if and only if
\begin{align*}
\left(\int_0^1\left(\int_0^tg^*(t)\,ds\right)^{\frac{q}{p}}\frac{dt}{t}\right)^{\frac{1}{q}} \le C \left(\int_0^1g^*(t)\left(\log \frac{e}{t}\right)^{\alpha p}\,dt\right)^\frac1p
\end{align*}
for every $g\in\M(\RR,\mu)$. In turn, on raising both sides of the last inequality to $p$, we see that~\eqref{EQ:>C} holds if and only if
\begin{align*}
\left(\int_0^1\left(\int_0^tg^*(t)\,ds\right)^{\frac{q}{p}}\frac{dt}{t}\right)^{\frac{p}{q}} \le C^{p} \int_0^1g^*(t)\left(\log \frac{e}{t}\right)^{\alpha p}\,dt
\end{align*}
for every $g\in\M(\RR,\mu)$. However, the last inequality represents the continuous embedding
\begin{equation}\label{E:embedding-classical-spaces-1}
    \Lambda^1(v)\hookrightarrow \Gamma^{\frac{q}{p}}(w),
\end{equation}
where $v(t)=(\log\frac{e}{t})^{\alpha p}$ and $w(t)=t^{\frac{q}{p}-1}$ for $t\in(0,1)$. Necessary and sufficient conditions for~\eqref{E:embedding-classical-spaces-1} to hold are known, see e.g.~\cite[Theorem~10.3.12(ii)]{PKJF}. After modifying the result to the case when $\mu(\RR)=1$, we obtain that~\eqref{E:embedding-classical-spaces-1} is true if and only if both the conditions
\begin{equation}\label{E:first-embedding-first-condition}
    \sup_{t\in(0,1)}\left(\int_{0}^{t}w(s)\,ds\right)^{\frac{p}{q}}\left(\int_{0}^{t}v(s)\,ds\right)^{-1}<\infty
\end{equation}
and
\begin{equation}\label{E:first-embedding-second-condition}
\sup_{t\in(0,1)}t\left(\int_{t}^{1}\frac{w(s)}{s^{\frac{q}{p}}}\,ds\right)^{\frac{p}{q}}\left(\int_{0}^{t}v(s)\,ds\right)^{-1}<\infty
\end{equation}
are satisfied. On inserting for $v$ and $w$ and performing a simple calculation we see that~\eqref{E:first-embedding-first-condition} amounts to
\begin{equation*}
    \sup_{t\in(0,1)}\left(\log\frac{e}{t}\right)^{-\alpha p}<\infty,
\end{equation*}
which is clearly true for any $\alpha\in(0,\infty)$, while~\eqref{E:first-embedding-second-condition} turns into
\begin{equation*}
    \sup_{t\in(0,1)}\left(\log\frac{e}{t}\right)^{\frac{p}{q}-\alpha p}<\infty,
\end{equation*}
which is satisfied if and only if $\alpha\in[\frac{1}{q},\infty)$. We have a twofold use for this result. First, on taking $\alpha=\frac{1}{q}$, we obtain that the first embedding in~\eqref{E:sandwich} is true. Second, it follows from here that the second embedding in~\eqref{E:sandwich} cannot be reversed.

Using the similar argumentation as in the first part of the proof we can verify that, given $\alpha\in(0,\infty)$, the embedding $Y_{p,q}\hookrightarrow L^{p,p;\alpha}$  is equivalent to
\begin{equation}\label{E:embedding-classical-spaces-2}
    \Gamma^{\frac{q}{p}}(w) \hookrightarrow \Lambda^1(v)
\end{equation}
with $v(t)=(\log\frac{e}{t})^{\alpha p}$ and $w(t)=t^{\frac{q}{p}-1}$ for $t\in(0,1)$. By an appropriate modification of~\cite[Theorem~10.3.17(ii)]{PKJF} (and observing that the non-degeneracy assumptions of that theorem are satisfied), we obtain that~\eqref{E:embedding-classical-spaces-2} is satisfied if and only if
\begin{equation*}
    \int_{0}^{1}
        \frac{t^{\frac{q}{q-p}+\frac{q}{p}-1}\left(\sup_{y\in(t,1)}\left(\log\frac{e}{y}\right)^{\frac{\alpha pq}{q-p}}\right)t^{\frac{q}{p}}\log\frac{1}{t}}
        {\left(t^{\frac{q}{p}}+t^{\frac{q}{p}}\log\frac{1}{t}\right)^{\frac{p}{q-p}+2}}\,dt<\infty.
\end{equation*}
Performing the calculation, we conclude that this condition is satisfied if and only if $\alpha\in(0,\frac{1}{q})$. Again, this tells us two things: the second embedding in~\eqref{E:sandwich} holds and, at the same time, the first embedding in~\eqref{E:sandwich} cannot be reversed. The proof is complete.
\end{proof}

\section{Associate space}

One of the most important problems concerning any rearrangement-invariant structure is the characterization of its associate space (see Definition~\ref{D:associate}). There is plenty of motivation for such research as duality techniques constitute generally an indispensable tool in many applications of function spaces.

In the case of $X^{\aaa}$ with general r.i.~space $X$, the task of nailing down its associate space can be very difficult, if not impossible. We shall show, however, that at least in the particular case when $X$ is a classical Lorentz space of type $\Lambda$, a characterization is possible. The idea is based on the fact that, using
an~appropriate change of variables, the norm in the associate space can be shown to be equivalent to a~power of the norm of a~certain continuous embedding between classical Lorentz spaces of types $\Gamma$ and $\Lambda$.

While the theory of spaces $X^{\aaa}$ developed in the preceding sections is restricted to the case when $X$ is an r.i.~space, for the investigation of associate spaces of the spaces $(\Lambda^{q}(w))^{\aaa}$ we can afford a
more
general approach. A restriction to the cases when $\Lambda^{q}(w)$ is an r.i.~space would narrow the field of examples to rather exceptional cases, since $\Lambda^{q}(w)$ is an r.i.~space only if it coincides with $\Gamma^{q}(w)$ - see~\cite{Sa,CGS,CPSS,CKMP,PKJF}. However, it turns out that such restriction is not necessary. The functionals $\varrho_{\Lambda^p(w)}$ are defined through the non-increasing rearrangement of a function, whence there is a
natural way of defining $(\varrho_{\Lambda^p(w)})^{\aaa}$ for any positive $\alpha$, regardless of whether the original functional is a norm or not. At the same time, in Definition~\ref{D:associate} we introduced the associate functional $\varrho'$ for any non-negative functional $\varrho$ acting on $\M_+(\RR,\mu)$.

\begin{defn}
Let $q\in(0,\infty)$, $\alpha\in(0,\infty)$ and $w\in\M_+(0,\mu(\RR))$. We then define the set $(\Lambda\sp q(w))^{\aaa}$ as the collection of all functions $f\in\M_0(\RR,\mu)$ such that $\varrho_{\Lambda\sp q(w)}^{\aaa}(|f|)<\infty$, in which
\begin{equation*}
    \varrho_{\Lambda\sp q(w)}^{\aaa}(f) = \left(\int_{0}^{\mu(\RR)}\left(\frac{1}{t}\int_0^{t}f^*(s)^{\alpha}\,ds\right)^{\frac{q}{\alpha}}w(t)\,dt\right)^{\frac{1}{q}}
\end{equation*}
for $f\in\M_+(\RR,\mu)$.
\end{defn}

Let us note that in the case when $\Lambda\sp q(w)$ is an r.i.~space the definition of $ \varrho_{\Lambda\sp q(w)}^{\aaa}$ coincides with the one given before. Now we are in a position to present the main result of this section.

\begin{thm}\label{T:classical-lambda}
Let $q\in(0,\infty)$, $\alpha\in(0,\infty)$ and let $w\in\M_+(0,b)$, where $b=\mu(\RR)$ (here $b$ can be either finite or infinite). Moreover, let $w$ satisfy the non-degeneracy conditions
\begin{equation*}
\int_0^{b} \frac{w(s)}{(s+1)^{\frac{q}{\alpha}}}\,ds<\infty,
\qquad
\int_0^1\frac{w(s)}{s^{\frac{q}{\alpha}}}\,ds
=
\infty,
\end{equation*}
and, in case when $b=\infty$, also
\begin{equation*}
\int_1^{\infty}w(s)\,ds=\infty.
\end{equation*}
Let $X=\Lambda\sp q(w)$.

\textup{(i)} If $q\in(0,1]$ and $\alpha\in(0,1]$, then
\[
\|g\|_{(X^{\aaa})'}=
\sup_{t\in(0,b)}
\frac{tg\sp{**}(t)}
{\left(\int_0\sp tw(s)\,ds+t\sp{\frac{q}{\alpha}}\int_t\sp{b}w(s)s\sp{-\frac{q}{\alpha}}\,ds\right)\sp{\frac{1}{q}}}.
\]

\textup{(ii)} If $q\in(1,\infty)$ and $\alpha\in(0,1]$, then
\[
\|g\|_{(X^{\aaa})'}\approx
\left(\int_0\sp{b}
\frac
{t\sp{\frac{q'+q}{\alpha}-1}
\sup_{y\in(t,b)}
y\sp{q'-\frac{q'}{\alpha}}
g\sp{**}(y)\sp{q'}
\int_0\sp{t}w(s)\,ds
\int_t\sp{b}w(s)s\sp{-\frac{q}{\alpha}}\,ds}
{\left(\int_0\sp tw(s)\,ds+t\sp{\frac{q}{\alpha}}\int_t\sp{b}w(s)s\sp{-\frac{q}{\alpha}}\,ds\right)\sp{q'+1}}
\,dt
\right)\sp{\frac{1}{q'}}.
\]

\textup{(iii)} If $q\in(0,1]$ and $\alpha\in(1,\infty)$, then
\[
\|g\|_{(X^{\aaa})'}\approx
\sup_{t\in(0,b)}
\frac
{tg\sp{**}(t)+t\sp{\frac{1}{\alpha}}
\left(\int_t\sp{b}g\sp{**}(s)\sp{\frac{1}{\alpha-1}}g\sp*(s)\,ds
\right)\sp{1-\frac{1}{\alpha}}}
{\left(\int_0\sp tw(s)\,ds+t\sp{\frac{q}{\alpha}}\int_t\sp{b}w(s)s\sp{-\frac{q}{\alpha}}\,ds\right)\sp{\frac{1}{q}}}.
\]

\textup{(iv)} If $q\in(1,\infty)$ and $\alpha\in(1,\infty)$, then
\[
\|g\|_{(X^{\aaa})'}\approx
\left(\int_0\sp{b}
\frac
{
\left(\left(tg\sp{**}(t)\right)\sp{\frac{\alpha}{\alpha-1}}
+
t\sp{\frac{1}{\alpha-1}}\int_t\sp{b}g\sp{**}(s)\sp{\frac{1}{\alpha-1}}g\sp*(s)\,ds\right)
\sp{\frac{q'(\alpha-1)}{\alpha}-1}
\left(tg\sp{**}(t)\right)\sp{\frac{1}{\alpha-1}}g\sp*(t)
}
{
\left(\int_0\sp tw(s)\,ds+t\sp{\frac{q}{\alpha}}\int_t\sp{b}w(s)s\sp{-\frac{q}{\alpha}}\,ds\right)\sp{q'-1}}
\,dt
\right)\sp{\frac{1}{q'}}.
\]
\end{thm}

\begin{proof}
Our point of departure will be the definition of the associate space. We have
\[
\|g\|_{(X^{\aaa})'}
=
\sup_{\|h\|_{X^{\aaa}}\leq1}\int_0\sp{b}g\sp*(t)h\sp*(t)\,dt,
\]
that is,
\begin{equation}\label{E:associate}
\|g\|_{(X^{\aaa})'}
=
\sup_{h\not\equiv 0}
\frac{
\int_0\sp{b}g\sp*(t)h\sp*(t)\,dt}
{\left(\int_0\sp{b}\left(\frac1t\int_0\sp th\sp*(s)\sp{\alpha}\,ds\right)\sp{\frac{q}{\alpha}}w(t)\,dt\right)\sp{\frac{1}{q}}}.
\end{equation}
As already mentioned in the proof of Theorem~\ref{T:boundedness-of-M}, for every $h\in\M(\RR,\mu)$ there exists $f\in\M(\RR,\mu)$ such that $h\sp*=(f\sp*)\sp{\frac{1}{\alpha}}$. Hence, using the substitution $h\sp*\mapsto(f\sp*)\sp{\frac{1}{\alpha}}$ in~\eqref{E:associate}, we conclude that
\[
\|g\|_{(X^{\aaa})'}
=
\sup_{f\not\equiv 0}
\frac{
\int_0\sp{b}g\sp*(t)f\sp*(t)\sp{\frac{1}{\alpha}}\,dt}
{\left(\int_0\sp{b}\left(\frac1t\int_0\sp tf\sp*(s)\,ds\right)\sp{\frac{q}{\alpha}}w(t)\,dt\right)\sp{\frac{1}{q}}}.
\]
Raising both terms in the ratio to $\alpha$ we get
\[
\|g\|_{(X^{\aaa})'}
=
\left(\sup_{f\not\equiv 0}
\frac{
\left(\int_0\sp{b}g\sp*(t)f\sp*(t)\sp{\frac{1}{\alpha}}\,dt\right)\sp{\alpha}}
{\left(\int_0\sp{b}\left(\frac1t\int_0\sp tf\sp*(s)\,ds\right)\sp{\frac{q}{\alpha}}w(t)\,dt\right)\sp{\frac{\alpha}{q}}}\right)\sp{\frac{1}{\alpha}}.
\]
This, however, can be represented as
\begin{equation}\label{E:representation}
\|g\|_{(X^{\aaa})'}
=
\left(\sup_{f\not\equiv 0}
\frac
{\|f\|_{\Lambda\sp{\frac{1}{\alpha}}(g^{*})}}
{\|f\|_{\Gamma\sp{\frac{q}{\alpha}}(w)}}\right)\sp{\frac{1}{\alpha}}.
\end{equation}
The quantity in brackets at the right hand side of~\eqref{E:representation} is equal to the operator norm of the continuous embedding
\begin{equation*}
    \Gamma\sp{\frac{q}{\alpha}}(w) \hookrightarrow \Lambda\sp{\frac{1}{\alpha}}(g^{*}).
\end{equation*}

Assume that $q\in(0,1]$ and $\alpha\in(0,1]$. Then, by~\cite[Theorem~10.3.17(i)]{PKJF}, we obtain that
\begin{equation*}
    \sup_{f\not\equiv 0}
    \frac
    {\|f\|_{\Lambda\sp{\frac{1}{\alpha}}(g^{*})}}
    {\|f\|_{\Gamma\sp{\frac{q}{\alpha}}(w)}}
        =
            \sup_{t\in(0,b)}
            \frac{\left(\int_{0}^{t}g^{*}(s)\,ds\right)^{\alpha}}
            {\left(\int_0\sp tw(s)\,ds+t\sp{\frac{q}{\alpha}}\int_t\sp{b}w(s)s\sp{-\frac{q}{\alpha}}\,ds\right)\sp{\frac{\alpha}{q}}}.
\end{equation*}
Raising both sides of the last relation to $\frac{1}{\alpha}$ and plugging it into~\eqref{E:representation}, we establish the assertion in the case~(i).

The proof is analogous in all the remaining cases and it is based on parts (ii), (iii) and (iv) of~\cite[Theorem~10.3.17]{PKJF}.
\end{proof}

\begin{rem}
As a~special case of Theorem~\ref{T:classical-lambda}, we obtain the characterization of $(X^{\aaa})'$ when $X$ is a Lorentz--Zygmund space.
\end{rem}

\begin{rem}
It is worth noticing that in some particular cases one can obtain the result of Theorem~\ref{T:classical-lambda} directly. Namely, if $X^{\aaa}=X$, then obviously $(X^{\aaa})'=X'$. This applies for example when $\alpha\in(0,1)$ and $q,w$ are such that $\Lambda\sp q(w)$ is an r.i.~space, since then we can use Remark~\ref{R:alpha}.
\end{rem}

\end{document}